\documentclass[11pt]{article}

\usepackage{amssymb,amsmath,latexsym}
\usepackage[dvips]{graphics}

\usepackage{ifthen}
\usepackage{mhequ}

  \def\eref#1{(\ref{#1})}

\def\prt{\partial}

\def\ol{\overline}
\def\var{{\mathop {\rm Var}}}

\def\CC{{\cal C}}
\def\EE{{\cal E}}
\def\FF{{\cal F}}
\def\BB{{\cal B}}

\def\DD{{\cal D}}
\def\LL{{\cal L}}

\def\GG{{\cal G}}

\def\CP{{\cal P}}
\def\CX{{\cal X}}

\def\E{{\bf E}}
\def\F{{\bf F}}
\def\P{{\bf P}}

\def\bQ{{\bf Q}}

\def\RR{\mathbb{R}}

\def\n{{\bf n}}
\def\v{{\bf v}}
\def\z{{\bf z}}
\def\cn{{\bf u}}
\def\1{{\bf 1}}

\def\<{\langle}
\def\>{\rangle}

\def\<{\langle}
\def\>{\rangle}

\def\pf{\noindent{\bf Proof.} }
\def\eps{{\varepsilon}}

\def\wt{\widetilde}
\def\wh{\widehat}

\def\qed{{\hfill $\Box$ \bigskip}}

\def\bfb{{\bf b}}

\numberwithin{equation}{section}

\newtheorem{thm}{Theorem}[section]

\newtheorem{remark}[thm]{Remark}

\newtheorem{corollary}[thm]{Corollary}

\newtheorem{theorem}[thm]{Theorem}
\newtheorem{proposition}[thm]{Proposition}
\newtheorem{assumption}[thm]{Assumption}

\makeatletter
\def\sqr#1#2{%
    \fboxsep0pt%
    \fboxrule#2pt%
    \fbox{\hbox to #1pt{\hss\vbox to #1pt{\vss}}}}
\def\qed{\unskip\nobreak\hfil\penalty50\hskip0.3em\hbox{}\nobreak
      \hfil\sqr{5.5}{0.2}\parfillskip=\z@ \finalhyphendemerits=0 \par}
\makeatother
\newenvironment{proof}[1][ ]{\ifthenelse{\equal{#1}{ }}{\def\MHpfName{.}}{\def\MHpfName{ #1.}}%
    \trivlist\item[\hskip\labelsep{\bf Proof\MHpfName}]\ignorespaces}%
    {\phantom{a}\nobreak\qed\endtrivlist}

\def\qed{{\hfill $\square$ \bigskip}}

\def\square{{\vcenter{\vbox{\hrule height.3pt
        \hbox{\vrule width.3pt height5pt \kern5pt
           \vrule width.3pt}
        \hrule height.3pt}}}}

\def\bD {{\mathbb D}}

 \def\bQ {{\mathbb Q}}

\def\wt{\widetilde}
\def\wh{\widehat}
\def\ol{\overline}
\def\E{{\mathbb E}}
\def\P{{\mathbb P}}

\def\angel#1{{\langle #1 \rangle}}
\def\bee{\begin{equation}}
\def\eee{\end{equation}}

\def\R{{\mathbb R}}
\def\E{{{\mathbb E}\,}}
\def\P{{\mathbb P}}
\def\Q{{\mathbb Q}}

\def\F{{\cal F}}

\def\qed{{\hfill $\square$ \bigskip}}

\def\eps{\varepsilon}

\def\angel#1{{\langle#1\rangle}}

\def\wt{\widetilde}
\def\ol{\overline}

\def\wh{\widehat}

\def\square{{\vcenter{\vbox{\hrule height.3pt
        \hbox{\vrule width.3pt height5pt \kern5pt
           \vrule width.3pt}
        \hrule height.3pt}}}}

\def\tfrac#1#2{{\textstyle {\frac{#1}{#2}}}}

\def\tlint{{- \kern-0.85em \int \kern-0.2em}}  
\def\dlint{{- \kern-1.05em \int \kern-0.4em}}  

\def\bD {{\mathbb D}}

 \def\bQ {{\mathbb Q}}

\begin{document}

\title{\bf Stationary
distributions for diffusions with inert drift
 \footnote{
Research supported in part by NSF Grants DMS-0601783 and
DMS-0600206 and by EPSRC Grant EP/D071593/1.
The second and fourth authors gratefully acknowledge
the hospitality and support of the Institut Mittag-Leffler, where part
of this research was done.
}}
\author{ {\bf Richard F. Bass}, \
{\bf Krzysztof Burdzy},  \ {\bf Zhen-Qing Chen} \ and \ {\bf
Martin Hairer}}

\maketitle

\begin{abstract}
Consider a reflecting diffusion in a domain in $\RR^d$ that
acquires drift in proportion to the amount of local time spent on
the boundary of the domain. We show that the stationary
distribution for the joint law  of  the position of the reflecting
process and the value of the drift vector has a product form.
Moreover, the first component is the symmetrizing measure on the
domain for the reflecting diffusion without inert drift, and the
second component has a Gaussian distribution. We also consider
processes where the drift is given in terms of the gradient of a
potential.

\end{abstract}

\begin{doublespace}

\section{Introduction}

This article is concerned with the higher dimensional version of a
one-dimensional model originally introduced by Knight \cite{K} and
studied in more detail in \cite{W1,W2}. Computer simulations
presented in \cite{BHP} led to the conjecture that the stationary
distribution for this higher dimensional process has a certain
interesting structure. We prove this conjecture, and moreover answer
questions about the stationary distribution left open in \cite{BHP}.

We start with a presentation of the model in a simple case, that of
\cite{BHP}. We consider a bounded smooth domain $D\subset \R^d$ and
reflecting Brownian motion $X_t$ in $\ol D$ with  drift $K_t$. Let
$B_t$ be $d$-dimensional Brownian motion, $\n(x)$ be the unit inward
normal vector of $D$ at $x\in \prt D$ and let $L_t$ be the local time
of $X$ on $\prt D$, that is, a nondecreasing one-dimensional process
with continuous paths that increases only when $X_t\in \prt D$. The
pair of processes $(X,K)$ has the following representation:
\begin{align*}
X_t &= X_0 + B_t + \int_0^t \n(X_s) \,dL_s
  + \int_0^t K_s \,ds, \\
K_t &= K_0 + \int_0^t \n(X_s)d L_s
\end{align*}
with $(X_t, K_t)\in \ol D\times \R^d$ for all $t\geq 0$. Note that
$X_t$ is reflecting Brownian motion in $\ol D$ with normal reflection
at the boundary and with drift $K_t$, $K_t$ is an $\R^d$-valued
process that represents the accumulated local time on the boundary in
the direction normal to the boundary, and the drift $K_t$ does not
change when $X_t$ is in the interior of $D$. We call the process $K$
``inert drift'' because it plays a role analogous to the inert
particle in Knight's original model \cite{K} in one dimension. The
main simulation result of \cite{BHP} suggests that a stationary
distribution for $(X,K)$ exists and has a product form, i.e., that
$X_t$ and $K_t$ are independent for each time $t$ under the
stationary distribution. Moreover, the first component of the
stationary distribution is the uniform probability measure on $D$. We
prove rigorously in this paper that this indeed holds, and we further
show that the second component of the stationary distribution is
Gaussian.

The product form of the stationary distribution was initially a
mystery to us, especially since the components $X$ and $K$ of the
vector $(X,K)$ are {\it not} Markov processes. There are models known
in mathematical physics where the stationary distribution of a Markov
process has a product form although each component of the Markov
process is not a Markov process itself. Examples may be found in
Chapter VIII of \cite{L}, in particular, Theorem 2.1 on page 380. As
we will see at the beginning of Section \ref{sect:potential}, the
product form of the stationary distribution in our model comes
naturally from a computation with infinitesimal generators.

The Gaussian nature of the stationary distribution for $K$ is already
known in the one-dimension\-al case; see \cite{W1,BW}. An interesting
phenomenon observed in this paper is that if the inward normal vector
field ${\bf n}$ in the equation for $K$ is replaced by $\Gamma \n$
for some constant symmetric positive definite matrix $\Gamma$, $(X,
K)$ continues to have a product form for the stationary distribution,
but this time the component $K$ has a Gaussian distribution with
covariance matrix $\Gamma$.

The main goal of this paper is to address the existence and
uniqueness of the stationary distribution of normally reflecting
Brownian motion with inert drift and to give an explicit formula
for the stationary distribution. We also consider a larger class
of reflecting diffusions, including what is sometimes known as
distorted reflecting Brownian motion---see Theorems \ref{T:wc.3}
and \ref{T:irr.1} below. Distorted reflected Brownian motion  is
the reflecting diffusion with generator $\frac1{2\rho}  \,  \nabla
(\rho \nabla)$ for a suitable function $\rho$.

We start by showing in Section \ref{sec:diff} the weak existence and
weak uniqueness of solutions to an SDE representing a large family of
diffusions with reflection and inert drift. More specifically, let
$\rho$ be a $C^2$ function on $\overline D$ that is bounded between
two positive constants and $A(x)=(a_{ij}(x))_{n\times n}$ be a
matrix-valued function on $\RR^d$ that is symmetric, uniformly
positive definite, and each $a_{ij}$ is bounded and $C^2$ on
$\overline D$. The vector $\cn (x) :=\frac12 A(x) \n (x)$ is called
the conormal vector at $x\in \partial D$. Let
$\sigma(x)=(\sigma_{ij}(x))$ be the positive symmetric square root of
$A(x)$.  For notational convenience, we use $\partial_i$ to denote
$\frac{\partial}{\partial x_i}$. For $\varphi \in C^2(\RR^d)$, let
\begin{equation}\label{e:1.1}
\LL \varphi (x):=
   \frac1{2\rho (x) }  \sum_{i,j=1}^d  \partial_i  \big( \rho (x)
   a_{ij} (x) \partial_j  \varphi  (x) \big).
\end{equation}
Let
${\bf b}(x)$ be the vector whose $k^{th}$ component is
  $$b_k(x)=\frac1{2\rho (x)} \sum_{i=1}^d  \partial_i
  ( \rho (x) a_{ik}(x)).
   $$
Let $B$ be standard $d$-dimensional Brownian motion and ${\bf v}$ a
bounded measurable vector field on $\partial D$. Consider the
following  diffusion process $X$ taking values in $\overline D$ such
that for all $t\geq 0$,
\begin{equation}\label{e:1.2}
\begin{cases}
dX_t=\sigma (X_t) \,dB_t+ \bfb  (X_t) \,dt +
\cn
 (X_t) \,dL_t+K_t\,dt, \\
 t\mapsto L_t \hbox{ is continuous and
    non-decreasing with } L_t= \int_0^t \1_{\partial D}(X_s) \,dL_s, \\
dK_t=  \v (X_t) \,dL_t.
\end{cases}
\end{equation}
In Theorem \ref{T:2.1}, we show that the above stochastic
differential equation (SDE) has a unique weak solution $(X, K)$ for
every starting point $(x_0, k_0)\in \overline D \times \RR^d$. The
solution $(X, K)$ of \eqref{e:1.2} is called a (symmetric) reflecting
diffusion on $D$ with  inert drift. Let $X^0$ be   symmetric
reflecting diffusion on $D$ with infinitesimal generator $\LL$ in
\eqref{e:1.1}; that is, $X^0$ is a continuous process taking values
in $\overline D$ such that for every $t\geq 0$,
\begin{equation}\label{e:1.3}
\begin{cases}
dX^0_t=\sigma (X^0_t) \,dB_t+ \bfb  (X^0_t) \,dt +
\cn (X^0_t) \,dL^0_t , \\
 t\mapsto L^0_t \hbox{ is continuous and
    non-decreasing with } L^0_t= \int_0^t \1_{\partial D}(X^0_s) \,dL^0_s .
\end{cases}
\end{equation}
The continuous non-decreasing process $L^0$ is called the boundary
local time of $X^0$. When $\sigma$ is the identity matrix, $X^0$ is
distorted reflecting Brownian motion on $\overline D$. The main
observation of Section \ref{sec:diff} is that the reflecting
diffusion $(X, K)$ with inert drift  can be obtained from the
reflecting diffusion $\left( X^0, K_0+\int_0^\cdot \n (X_s^0)\,dL^0_s
\right)$ without inert drift  by a suitable Girsanov transform, and
vice versa.

The questions of strong existence and strong uniqueness for solutions
to \eqref{e:1.2} are discussed in Section \ref{sec:pu}. They are
resolved positively under the  additional assumption that $\v =a_0
{\bf u}$ for some constant $a_0\in \RR$. This section uses some ideas
and results from \cite{LS}, but the main idea of our argument is
different from the one in that paper, and we believe ours is somewhat
simpler.

In Section \ref{sect:potential} we consider symmetric diffusions with
drift given as the gradient of a potential. We do this  because the
analysis of the stationary distribution is much easier in the case of
a smooth potential than the ``singular'' potential representing
reflection on the boundary of a domain. More specifically, let
$\Gamma$ be a symmetric positive definite constant   $d\times
d$-matrix and $V\in C^1(D)$ tending to infinity in a suitable way as
$x$ approaches the boundary $\partial D$. Consider the following
diffusion process $X$ on $D$ associated with  generator $\LL =
\frac12 e^V \left( e^{-V} A\nabla \right)$ but with an additional
``inert'' drift $K_t$:
\begin{equs}\label{e:1.4}
\begin{cases}
dX_t &= \sigma (X_t) \,dB_t +{\bf b}(X_t) \,dt -\frac12
(A\nabla V)(X_t)\, dt + K_t\, dt \;, \\
dK_t &= -\frac12 \Gamma \, \nabla V (X_t)\, dt\;,
 \end{cases}
\end{equs}
Here $A=A(x)=\sigma^T \sigma$ is a $d\times d$ matrix-valued function
that is uniformly elliptic and bounded, with $\sigma_{ij} \in C^1
(\overline D)$, and ${\bf b}=(b_1, \cdots, b_d)$ with
$b_k(x):=\frac12 e^{V(x)} \sum_{i=1}^d
\partial_i \left( e^{-V(x)}a_{ik}(x) \right)$.
We show in Theorem \ref{T:4.3} that if $e^{-V/2}\in W^{1,2}_0(D)$,
then the SDE \eqref{e:1.4} has a unique conservative solution $(X, K)$ which
has a stationary probability distribution
$$\pi (dx, dy)= c_1 \1_D (x) e^{ -V(x)}  \,
e^{-(\Gamma^{-1}y, y)} \, dx \,dy.
$$
In other words, $(X, K)$ has a product form stationary distribution,
with $c \, e^{-V}\,dx$ in the $X$ component, and a Gaussian
distribution with covariance matrix $\Gamma$ in the $K$ component.
Observe that $c\, \, e^{-V}\,dx$ with normalizing constant $c>0$ is
the stationary distribution for the conservative symmetric diffusion
$X^0$ in $D$ with generator $\frac12 e^V \nabla (\nabla e^{-V} A
\nabla )$. The uniqueness of the stationary measure for solutions of
\eqref{e:1.4} is addressed in Proposition \ref{prop:unique} under
much stronger conditions.

In Section \ref{sec:weakconv}, we prove weak convergence of a
sequence of solutions of \eqref{e:1.4} to symmetric reflecting
diffusions with inert drift given by \eqref{e:1.2} where $\v=\Gamma
\, \n$ for some symmetric positive definite constant matrix. This
implies
$$ \pi(dx, dy)=c_2 \1_D(x) \rho (x) \, e^{-(\Gamma^{-1}y, y )} \, dx\,dy
$$
is a stationary distribution for solutions $(X, K)$ of \eqref{e:1.2}
with $\v =\Gamma \n$; see Theorem \ref{T:wc.3}. Observe that $c_3 \,
\rho (x) \,dx$, where $c_3>0$ is a  normalizing constant, is the
stationary distribution for the symmetric reflecting diffusion $X^0$
on $\overline D$ with generator $\frac1{2\rho }\, \nabla (\rho A
\nabla)$ in \eqref{e:1.3}.

Finally, Section \ref{sec:irreducibility} completes our program by
showing irreducibility in the sense of Harris for reflecting
diffusions with inert drift given by \eqref{e:1.2} under the
assumption that $\v=\Gamma \, \n$ for some symmetric positive
definite constant matrix,  $A $ is the identity matrix,   ${\bf u} =
\n$,  and ${\bf b}  =  \frac12 \nabla \log \rho$. Uniqueness of the
stationary distribution follows from the irreducibility.

The level of generality of our results varies throughout the paper,
for technical reasons. We leave it as an open problem to prove a
statement analogous to Theorem \ref{T:irr.1} in the general setting
of Theorem \ref{T:2.1}. The calculation at the beginning of Section
\ref{sect:potential} indicates that in order for the solution $(X,
K)$ of \eqref{e:1.2} to have a product form stationary distribution,
the inert drift vector field $\v$ has to be of the form $\Gamma \,
\n$ for some symmetric positive definite constant matrix $\Gamma$. We
leave its verification as another open problem.

\medskip

Our model belongs to a family of processes with
``reinforcement'' surveyed by Pemantle in \cite{Pem}; see
especially Section 6 of that survey and references therein.
Papers \cite{BenLR, BenR1, BenR2} study a process with a drift
defined in terms of a ``potential'' and the ``normalized''
occupation measure. While there is no direct relationship to
our results, there are clear similarities between that model
and ours.

We are grateful to Boris Rozovsky and David White for very useful advice.

\section{Weak existence and uniqueness}
\label{sec:diff}

This section is devoted to weak existence and uniqueness of
solutions to an SDE representing a family of reflecting
diffusions with inert drift.

Let $D$ be a bounded $C^2$ domain in $\RR^d$, $d\geq 2$, and
denote by $\n$ the inward unit normal vector field on $\partial
D$. Throughout this paper, all vectors are column vectors. Let
$\rho$ be a $C^2$ function on $\overline D$ that is bounded
between two positive constants and $A(x)=(a_{ij}(x))_{n\times n}$
be a matrix-valued function on $\RR^d$ that is symmetric,
uniformly positive definite, and each $a_{ij}$ is bounded and
$C^2$ on $\overline D$. The vector $\cn (x) :=\frac12 A(x) \n (x)$
is called the conormal vector at $x\in
\partial D$. Clearly there exists $c_1 >0$ such that
$\cn(x) \cdot \n(x) > c_1$ for all $x\in \prt D$. Let
$\sigma(x)=(\sigma_{ij}(x))$ be the positive symmetric square root
of $A(x)$. For $\varphi \in C^2(\RR^d)$, let
$$\LL \varphi (x):=
   \frac1{2\rho (x) }  \sum_{i,j=1}^d \partial _i \left( \rho (x)
   a_{ij} (x) \partial_j \varphi (x) \right).
$$
Let $\bfb(x)$ be the vector whose $k^{th}$ component is
  $$b_k(x)=\frac1{2\rho (x)} \sum_{i=1}^d \partial_i( \rho (x) a_{ik}(x)). $$
 Thus $b_k$
is the same as $\LL$ operating on the function $f_k(x)=x_k$.

Let $B$ be standard $d$-dimensional Brownian motion
and ${\bf v}$ a bounded measurable vector field on $\partial D$.
Consider the following system of stochastic differential equations,
with the extra condition that $X_t \in \overline D$ for all $t\geq
0$:
\begin{equation}\label{eqn:1}
\begin{cases}
dX_t=\sigma (X_t) \,dB_t+ \bfb  (X_t) \,dt +
\cn
 (X_t) \,dL_t+K_t\,dt, \\
 t\mapsto L_t \hbox{ is continuous and
    non-decreasing with } L_t= \int_0^t \1_{\partial D}(X_s) \,dL_s, \\
dK_t=  \v (X_t) \,dL_t.
\end{cases}
\end{equation}

\medskip

The proof of the next theorem says that   reflecting diffusions
with inert drift can be obtained from the corresponding symmetric
reflecting diffusions without inert drift by suitable Girsanov
transforms, and vice versa.

\begin{thm}\label{T:2.1}
For every $x\in \overline D$ and $y\in \RR^d$ there exists a
unique weak solution $\{(X_t, K_t), t\in[0,\infty)\}$ to
(\ref{eqn:1}) with $(X_0, K_0)=(x, y)$.
\end{thm}

\begin{proof}
Consider the following SDE,
\begin{equation}\label{eqn:2.2}
\begin{cases}
dX_t=\sigma (X_t) \,dB_t+ \bfb (X_t) \,dt +
\cn  (X_t)\,dL_t , \\
 t\mapsto L_t \hbox{ is continuous and
    non-decreasing with } L_t= \int_0^t \1_{\partial D}(X_s) \,dL_s,
    \end{cases}
\end{equation}
with $X_t\in \overline D$ for every $t\geq 0$. Weak existence and
uniqueness of solutions to (\ref{eqn:2.2}) follow from \cite{C} or
\cite{DI}.
(We have been informed by the authors  of \cite{DI} that there is
a gap in the proof of Case 2 in \cite{DI}, but we need only Case 1. Moreover they
have informed us that a correction for Case 2 is in press.)
  The distribution of the solution to (\ref{eqn:2.2}) with
$X_0 = x \in D$ will be denoted by $\P_x$.

Note that the remaining part of the proof uses only the
$C^1$-smoothness of the domain and Lipschitz continuity of
$a_{ij}$ and $\rho$. Let $K_t:=y+\int_0^s \v (X_s) \,dL_s$ and
$\sigma^{-1}(x)$ be the inverse matrix of $\sigma (x)$. Define for
$t\geq 0$,
$$ M_t=\exp\left( \int_0^t \sigma^{-1} (X_s) K_s \,dB_s-\frac 12 \int_0^t
|\sigma^{-1} (X_s) K_s|^2 \,ds \right).
$$
It is clear that $M$ is a continuous positive local martingale
with respect to the minimal augmented filtration $\{\FF_t,
t\geq 0\}$ of $X$. Let $T_n = \inf\{t>0: |K_t| \geq 2^n\}$.
Since $L_t <\infty$ for every $t<\infty$, $\P_x$-a.s., and
$|\v|$ is uniformly bounded, we see that $K_t <\infty$ for
every $t<\infty$, $\P_x$-a.s. Hence, $T_\infty := \lim_{n\to
\infty}T_n=\infty$, $\P_x$-a.s. For every $n\geq 1$,
$\{M_{T_n\wedge t}, \FF_t, t\geq 0\}$ is a martingale.

For $x\in \overline D$, $y\in \RR^d$ and $n\geq 1$, define a
new probability measure $\bQ_{x,y}$ by
$$ d \bQ_{x,y} = M_{T_n}
 \, d\P_x \qquad \hbox{on } \FF_{T_n}  \hbox{ for every } n\geq
1.
$$
It is routine to check this defines a probability measure
$\bQ_{x, y}$ on $\FF_\infty$. By the Girsanov theorem (cf.
\cite{RY}), the process
$$W_t:= B_t -\int_0^t \sigma^{-1} (X_s) K_s \,ds,
$$
is a Brownian motion up to time $T_n$  for every $n\geq 1$, under
the measure $\bQ_{x,y}$. Thus we have from (\ref{eqn:2.2}) that
under $\bQ_{x,y}$, up to time $T_n$  for every $n\geq 1$,
$$ dX_t=\sigma (X_t) \,dW_t+ \bfb (X_t) \,dt +
\cn  (X_t)\,dL_t + K_t \,dt .
$$
In other words, $\{(X_t, K_t), 0\leq t < T_\infty\}$ under the
measure $\bQ_{x,y}$ is a weak solution of (\ref{eqn:1}).

\bigskip

We make a digression on the use of the strong Markov property.
At this point in the proof, we cannot claim that $(X,K)$ is a
strong Markov process. Note however that if $T$ is a finite
stopping time, then $(X_{t+T}, K_{t+T})$ is again a solution of
\eqref{eqn:1} with initial values $(X_T,K_T)$. We can therefore
use regular conditional probabilities as a technical substitute
for the strong Markov property; this technique has been
described in great detail in Remark 2.1 of \cite{BBC}, where we
used the name ``pseudo-strong Markov property.'' Throughout the
remainder of this proof we will use the pseudo-strong Markov
property in place of the traditional strong Markov property and
refer the reader to \cite{BBC} for details.

\medskip

We will next show that $T_\infty=\infty$, $\bQ_{x,y}$-a.s., i.e.,
the process is conservative. This is the same as saying $|K|$ does
not ``explode'' in  finite time under $\bQ_{x, y}$. The intuitive
reason why this should be true is the following. Consider a
one-dimensional
 Brownian motion starting at $1$ with a
very large constant
 negative drift of size
$c$
 and reflect it at the origin. Then a simple calculation shows
that the local time accumulated at the origin by this process up
to time $1$ is also of order
$c$.
 This suggests that if $K$ is very large, the local time
accumulated in a time interval of order $1$ by $X$ on the
boundary $\prt D$ will be approximately proportional to $K$. Since
this feeds back into the right hand side of
the definition of $K$ in \eqref{eqn:1},
 one would expect $K$ to grow at most exponentially fast in
time.

Consider $\eps=2^{-j}>0$ where $j\geq 1$ is an integer. Our
argument applies only to small $\eps>0$ so we will now impose some
assumptions on $\eps$. Consider $x_0 \in \prt D$ and let
$CS_{x_0}$ be an orthonormal coordinate system such that $x_0=0$
in $CS_{x_0}$ and the positive part of the $d$-th axis contains
$\n(x_0)$. Let $\n_0 = \n(x_0)$. Recall that $D$ has a
$C^1$ boundary and that there exists $c_1 >0$ such that
$\cn(x) \cdot \n(x) > c_1$ for all $x\in \prt D$. Hence there
exist $\eps_0>0$ and $c_2>0$, such that for all $\eps\in
(0,\eps_0)$, every $x_0\in \prt D$ and all points $x=(x_1, \dots ,
x_d) \in \prt D \cap B(x_0, (6/c_2 +5)\eps)$, we have $|x_d| <
\eps/2$ and $\cn(x) \cdot \n_0 > c_2$, in $CS_{x_0}$. Since
$|\v(x)| \leq c_3 < \infty$ for all $x\in \prt D$, we can make
$\eps_0>0$ smaller, if necessary, so that $(2c_2)/(c_3 \eps) -
5\eps \geq \eps$ for all $\eps\in (0,\eps_0]$.

Let
\begin{equs}
 S_0 &= \inf\{t>0: |K_t| \geq 1/\eps\} , \\
 S_{n+1} &= \inf\{t> S_{n}: |\, |K_t|-|K_{S_{n}}|\,| \geq (6c_3/c_2)\eps\},
 \qquad n\geq 0.
\end{equs}
We will estimate $\bQ_{x,y} (S_{n+1} - S_{n} > \eps^2 \mid
\FF_{S_n})$ for $n=0, \, \dots,  \, [\eps^{-2}c_2/(6c_3)]$. Note
that $X_{S_n} \in \prt D$ for every $n$ because $K$ does not
change when $X$ is in the interior of the domain. Note also that
$|K_{S_n}| \le 2/\eps$ for every $n \le \eps^{-2}c_2/(6c_3)$ by
construction.

For $n\geq 0$, let
\begin{equs}
Y^{(n)}_t &= \int_{S_n}^{S_n+t} \sigma (X_s) \,dW_s+
\int_{S_n}^{S_{n}+t} \bfb (X_s) \,ds,\\
F_n &=\left\{
 \sup _{t\in[0, \eps^2]} |Y^{(n)}_t| < \eps \right\}.
\end{equs}
It is standard to show that there exist $\eps_0, p_0>0$, not depending
on $n$, such that if $\eps<\eps_0$, then
$$\bQ_{x,y} (F_n \mid \FF_{S_n}) \geq p_0.
$$
Let
 $$R_n = (S_n+\eps^2) \land \inf\{t \geq S_n:
 |K_t| \geq 4/\eps\}.
 $$

Suppose that the event $F_n$ holds. We will analyze the path of the process
$\{(X_t,K_t), S_n \leq t \leq R_n\}$. First, we will argue that
$X_t \in B( X_{S_n}, (6/c_2 + 5)\eps)$ for $S_n \leq t \leq R_n$.
Suppose otherwise. Let $U_1 = \inf\{t> S_n: X_t \notin B( X_{S_n},
(6/c_2 + 5)\eps)\}$ and $ U_2 = \sup\{t< U_1: X_t \in \prt D\}$.
We have
 $$(6/c_2 + 5)\eps =|X_{U_1} - X_{S_n}|
 = \left | Y_{U_1 - S_n} + \int_{S_n}^{U_1} K_t \,dt
 + \int_{S_n}^{U_1} \cn(X_t) \,dL_t \right| .
 $$
So on $F_n\cap \{ U_1 \leq R_n\}$,
\begin{equs}
 \left|\int_{S_n}^{U_2} \cn(X_t) \,dL_t \right| &=
 \left|\int_{S_n}^{U_1} \cn(X_t) \,dL_t \right| \\
& \geq (6/c_2 + 5) \eps - \left | Y_{U_1 - S_n} + \int_{S_n}^{U_1} K_t\, dt\right| \\
 &\geq (6/c_2 + 5) \eps - \eps - \int_{S_n}^{S_n+\eps^2} 4/\eps \,
 dt \\
 &\geq (6/c_2 + 5)\eps - \eps - 4\eps \\
 &= (6/c_2) \eps.
\end{equs}
We will use the coordinate system $CS_{X_{S_n}}$ to make the
following observations. The last formula implies that the
$d$-th coordinate of $\int_{S_n}^{U_2} \cn(X_t) \,dL_t$ is not
less than $6\eps$. Hence the $d$-th coordinate of $X_{U_2}$
must be greater than or equal to
 $$6 \eps -  \left | Y_{U_1 - S_n} + \int_{S_n}^{U_1} K_t \,dt\right|
 \geq 6 \eps - \eps - \int_{S_n}^{S_n+\eps^2} 4/\eps \,dt
 \geq 6\eps - \eps - 4\eps = \eps.
 $$
This is a contradiction because $X_{U_2}\in \prt D$ and the $d$-th
coordinate for all $x\in \prt D \cap B( X_{S_n}, (6/c_2 + 5)\eps)$
is bounded by $ \eps/2$. So on $F_n$, we have $X_t \in B(X_{S_n},
(6/c_2 + 5)\eps)$ for $t\in [S_n, R_n]$.

We will use a similar argument to show that $R_n = S_n + \eps^2$ on
$F_n$. Suppose that $R_n < S_n + \eps^2$. Then $S_n \leq R_n$ and
$|K_{R_n}|=4/\eps$. Let $U_3 = \sup\{t< R_n: X_t \in \prt D\}$. The
definition of $S_n$ implies that $K_{S_n} \leq 2/\eps$ for $n\leq
[\eps^{-2}c_2/(6c_3)]$. We have
 $$\left|\int_{S_n}^{R_n} \v(X_t) \,dL_t\right|
 = |K_{R_n } - K_{S_n}|\geq 4/\eps - 2/\eps = 2/\eps.
 $$
Since $|\v(x)|\leq c_3$, we have $L_{R_n} - L_{S_n} \geq
 2/(c_3\eps)$, so the $d$-th
coordinate of $\int_{S_n}^{R_n} \cn(X_t) \,dL_t$, which is the same as
$\int_{S_n}^{U_3} \cn(X_t) \,dL_t$, is bounded below by $(2c_2)/(c_3
\eps)$. But then on $F_n$, the $d$-th coordinate of $X_{U_3}$ must
be greater than or equal to
 $$(2c_2)/(c_3 \eps) -  \left | Y_{R_n - S_n} + \int_{S_n}^{R_n} K_t \,dt\right|
 \geq (2c_2)/(c_3 \eps) - \eps - \int_{S_n}^{S_n+\eps^2} 4/\eps \,dt
 =(2c_2)/(c_3 \eps) - 5\eps \geq \eps.
 $$
This is a contradiction because $X_{U_3}\in \prt D$ and the $d$-th
coordinate for all $x\in \prt D \cap B(x_0, (6/c_2 + 5)\eps)$ is
bounded by $ \eps/2$. So on $F_n$ we have $R_n = S_n + \eps^2$.

We will now use the same idea to show that on $F_n$, $L_{S_n+\eps^2}
- L_{S_n} < (6/c_2)\eps$. Assume that $L_{S_n+\eps^2} - L_{S_n} \geq
(6/c_2)\eps$. Let $U_4 = \sup\{t\leq S_n+\eps^2: X_t \in \prt D\}$.
The $d$-th coordinate of $\int_{S_n}^{U_4} \cn(X_t) \,dL_t$ is bounded
below by $6 \eps$. But then on $F_n$, the $d$-th coordinate of
$X_{U_3}$ must be greater than or equal to
 $$6 \eps -  \left | Y_{R_n - S_n} + \int_{S_n}^{R_n} K_t \,dt\right|
 \geq 6 \eps - \eps - \int_{S_n}^{S_n+\eps^2} 4/\eps \,dt
 \geq \eps.
 $$
This is a contradiction because $X_{U_3}\in \prt D$ and the $d$-th
coordinate for all $x\in \prt D \cap B(x_0, (6/c_2 + 5)\eps)$ is
bounded by $ \eps/2$. We see that if the event $F_n$ holds, then
 $$|\,|K_{S_n+\eps^2}| - |K_{S_n}|\,|\leq
 \int_{S_n}^{S_n+\eps^2} |\v(X_t)| \,dL_t
 \leq c_3 (L_{S_n+\eps^2} - L_{S_n})
 \leq  c_3 (6/c_2)\eps.
 $$
Hence, if the event $F_n$ holds, then $S_{n+1} > S_n +\eps^2$. We see that
$$\bQ_{x,y} (S_{n+1} > S_n +\eps^2 \mid \FF_{S_n}) \geq p_0.
$$
Recall that we took $\eps = 2^{-j}$ so that $S_0 = T_j$. The
last estimate and the pseudo-strong Markov property applied at
stopping times $S_n$ allow us to apply some estimates known for
a Bernoulli sequence with success probability $p_0$ to the
sequence of events $\{S_{n+1} > S_n +\eps^2\}$. Specifically,
there is some $p_1>0$ so that for all sufficiently small
$\eps=2^{-j}>0$,
\begin{equs}
\bQ_{x,y} (T_{j+1} &- T_j \geq (c_2/(6c_3))p_0/2 \, \big| \,
\FF_{T_j}) \\
 &= \bQ_{x,y} \left( \sum_{n=0}^
{[\eps^{-2}c_2/(6c_3)]}
  (S_{n+1} - S_n)
  \geq (c_2/(6c_3))p_0/2 \, \Big| \,
\FF_{T_j}\right) \\
&\geq  \bQ_{x,y} \left( \frac{\eps^2}{c_2/(6c_3)}
\sum_{n=0}^
{[\eps^{-2}c_2/(6c_3)]}
   {\bf 1}_{\{S_{n+1} > S_n +\eps^2\}}
 \geq p_0/2 \, \Big| \, \FF_{T_j}\right) \\
 &\geq  p_1.
\end{equs}
Once again, we use an argument based on comparison with a
Bernoulli sequence, this time with success probability $p_1$.
We conclude that there are infinitely many $n$ such that $
T_{j+1} - T_j \geq (c_2/(6c_3)) p_0/2$, $\bQ_{x,y}$-a.s. We
conclude that $T_\infty = \infty$, $\bQ_{x,y}$-a.s., so our
process $(X_t,K_t)$ is defined for all $t\in [0,\infty)$.

Next we will prove weak uniqueness. Suppose that $\bQ'_{x,y}$ is
the distribution of any weak solution to (\ref{eqn:1}) and define
$\P'_x$ by
$$ d\P'_x = \frac{1}{M_{T_n}}\, d \bQ'_{x,y}
 \qquad \hbox{on } \FF_{T_n}
 \hbox{ for every } n\geq 1.
$$
Reversing the argument in the first part of the proof, we conclude
that $X$ under $\P'_x$ solves (\ref{eqn:2.2}). It follows from the
strong uniqueness for (\ref{eqn:2.2}) that $\P'_x = \P_x$, and,
therefore, $\bQ'_{x,y} = \bQ_{x,y}$ on $\FF_{T_n}$. Since this holds
for all $n$, we see that $\bQ'_{x,y} = \bQ_{x,y}$ on $\FF_\infty$.
\end{proof}

\begin{remark}\label{R:2.2}{\rm
The assumptions that $D$ is a bounded $C^2$ domain and that the
$a_{ij}$'s and $\rho$ are $C^2$ on $\overline D$ are only used in
the weak existence and uniqueness of the solution $X$ to
\eqref{eqn:2.2}. The remaining proof only requires that $D$ be a
bounded $C^1$ domain and that the $a_{ij}$'s and $\rho$ are
Lipschitz in $\overline D$. In fact, when $D$ is a   bounded $C^1$
domain and the $a_{ij}$'s and $\rho$ are Lipschitz on $\overline
D$, the weak existence of solutions to (\ref{eqn:2.2}) follows
from \cite{C} and so we have the weak existence to the SDE
\eqref{eqn:1}. We believe the weak uniqueness for solutions to
\eqref{eqn:2.2} and consequently to \eqref{eqn:1} also holds under
this weaker assumption by an argument analogous to that in
\cite[Section 4]{BBC2}. However to give the full details of the
proof would take a significant number of pages, so we leave the
details to the reader. \qed }
\end{remark}

\bigskip

We remarked earlier that if $T$ is a finite stopping time, then
$(X_{t+T}, K_{t+T})$ is again a solution to \eqref{eqn:1} with
starting point $(X_T,K_T)$. This observation together with the
weak uniqueness of the solution to \eqref{eqn:1} implies that
$(X_t,K_t)$ is a strong Markov process in the usual sense;
cf.~\cite[Section I.5]{Bass97}.

\section{Pathwise uniqueness}
\label{sec:pu}

We will prove strong existence and strong uniqueness for solutions
to \eqref{eqn:1} under assumptions stronger than those in Section
\ref{sec:diff}, namely, we will assume that the vector field ${\bf
v}$ is a fixed constant multiple of ${\bf u}$. Throughout this
section, $D$ is a bounded $C^2$ domain in $\RR^d$, each $a_{ij}$
and $\rho$ are $C^{1,1}$ on $\overline D$ and so the vector ${\bf
b}$ in \eqref{eqn:1} is Lipschitz continuous. Our approach to the
strong existence and uniqueness  for solutions to \eqref{eqn:1}
uses some ideas and results from \cite{LS}, but the main idea of
our argument is different from the one in that paper, and we
believe ours is somewhat simpler. It is probably possible to
produce a proof along the lines of \cite{LS}, but a detailed
version of that argument adapted to our setting would be at least
as long as the one we give here.

\begin{thm}\label{T:strong}
Suppose that ${\bf v}\equiv a_0 {\bf u}$ for some constant $a_0\in
\RR$. For each $(x,y) \in \overline D \times \RR^d$, there exists a
unique strong solution $\{(X_t, K_t), t\in[0,\infty)\}$ to
(\ref{eqn:1}) with $(X_0, K_0)=(x, y)$.
\end{thm}

\begin{proof}
When $a_0=0$, $K_t=K_0$ for every $t\geq 0$. In this case, the
result follows from \cite{DI} (as mentioned above there is a gap
in the proof of Case 2 in \cite{DI}; we need only Case 1). So
without loss of generality, we assume $a_0\not=0$.

First we will prove pathwise uniqueness. Suppose that there exist
two solutions $(X_t, K_t)$ and $(X'_t, K'_t)$, driven by the same
Brownian motion $B$, starting with the same initial values $(X_0,
K_0) = (X'_0, K'_0)=(x, y)$, and such that $(X_t, K_t) \ne (X'_t,
K'_t)$ for some $t$ with positive probability.

According to Lemma 4.1 of \cite{LS}, there exists a matrix-valued
function $\Lambda (x) = \{\lambda_{ij}(x)\}_{1\leq i,j \leq d}$,
$x\in \RR^d$, such that $x\to \Lambda(x)$ is  uniformly elliptic
and in $C^2_b$ and such that the following formulas \eqref{eqf:5} and
\eqref{eqf:1}  hold. We
have
\begin{equation}\label{eqf:5}
 \cn(x)^T \Lambda(x) = \n(x)^T \qquad \hbox{for }
 x\in \prt D.
\end{equation}
 Moreover, there exists $c_1<\infty$ such that
\begin{equation}\label{eqf:1}
 c_1 |x-x'|^2 + \cn(x)^T \Lambda(x) (x'-x) \geq 0
 \qquad \hbox{for } x\in \prt D \hbox{ and } x'\in \ol D.
\end{equation}
 In our setting, we can take
$\Lambda = 2A^{-1}$. Let $c_2>1$ be such that
\begin{equation}\label{eqn:c1}
 c_2^{-1} I_{d\times d} \leq \Lambda(x)\leq
c_2 I_{d\times d} \qquad \hbox{for every } x\in \RR^d,
\end{equation}
 where $I_{d\times d}$ is the  $d\times d$-dimensional identity matrix.
 Define $\lambda:=c_2^3$.

 Let
 \begin{equs}
 U_t &= \int_0^t \sigma(X_t)  \,dB_t
 + \int_0^t \bfb (X_t) \,dt,\\
 V_t &= \int_0^t  K_t \,dt,
 \end{equs}
and define $U'_t$ and $ V'_t$ in a similar way relative to
$X'$.

We fix an arbitrary $p_1<1$ and an integer $k_0$ such that in $k_0$
Bernoulli trials with success probability $1/2$, at least $k_0/4$ of
them will occur with probability $p_1$ or greater. Let $t_1=1/k_0$.

Consider some
$ \eps ,c_0>0$,
 and let
 \begin{equs}
 T_1 &= \inf\{t>0: |X_t - X'_t|
 >0
 \hbox{  or  } |U_t -U'_t|
 >0
 \hbox{  or  } |K_t -K'_t|
 >0
 \},\\
 T_k &= (T_{k-1} + t_1) \land
 \inf\{t>T_{k-1}: |X_t - X'_t| \vee  |U_t -U'_t| \geq
 \lambda^{k-1} \eps \}\\
 & \qquad \land\ \inf\{t>T_{k-1}:
 |L_t - L_{T_{k-1}}| \lor |L'_t - L'_{T_{k-1}}|> c_0
 \}, \qquad k\geq 2.
 \end{equs}
We will specify the values of
 $\eps$ and $c_0$ later in the proof.

It is easy to see from (\ref{eqn:1}) that it is impossible to have
$X_t$ equal to $ X'_t$ on some random interval $[0,T^*]$, and at
the same time $U_{t} \ne U'_{t}$ or $K_{t} \ne K'_{t}$ for some
$t\in[0,T^*]$. Hence, with probability 1, for every $\delta >0$
there exists $t_*\in (T_1, T_1+\delta]$ such that $X_{t_*} \ne
X'_{t_*}$.

The idea of our pathwise uniqueness proof is as follows. Define
 $$
S = (T_1 + 1/4)\land \inf\{t\geq T_1: (L_t - L_{T_1})\vee (L'_t -
L'_{T_1}) \geq c_0/(4t_1)\}.
 $$
We will show that for
any $p_1<1$ and integer $k_0=k_0(p_1)$ as defined above and every
$\eps>0$,
 $$\P \left( \sup_{t\in [0, S]}|X_t - X'_t| \leq \eps \lambda^{k_0} \right)\geq p_1.
 $$
 As $\eps>0$ is arbitrary, we have
 $$  \P \left(  X_t = X'_t \hbox{ for every } t\in [0, S] \right)\geq
 p_1>0,
 $$
which contradicts the  definition of
$T_1$, in view of the remarks following the definition.
 This contradiction proves the pathwise uniqueness.

\medskip

Gronwall's inequality says that if $g(t)$ is nonnegative and
locally bounded and $g(t)\leq a+b\int_0^t g(s)\, ds$, then
$g(t)\leq ae^{bt}$. Suppose now that $f$ is a nonnegative
nondecreasing function and $g(t)\leq f(t)+b\int_0^t g(s)\, ds$.
Applying Gronwall's inequality for $t\leq t_1$ with $a=f(t_1)$,
we have
\begin{equation}\label{eqf:2}
 g(t_1)\leq f(t_1)e^{bt_1}.
\end{equation}
We apply the inequality with
$g(t)=|K_{t}-K'_{t}|$
 and
 $$f(t)=
 |K_{T_1}-K'_{T_1}| +
 \sup_{T_1\leq s\leq t} \left(|X_s-X'_s| +|U_s-U'_s|\right).
 $$
 Since ${\bf v} \equiv a_0 {\bf u}$, we have
 $\int_0^t {\bf u}(X_s) \,dL_s =\frac1{a_0}(K_t-K_0)$ and so
 $$K_t-K_0=a_0 (X_t-x_0-U_t-V_t),$$
and similarly for $K'$. Thus
$$
K_t-K'_t=
 a_0(X_t-X'_t-U_t+U'_t-V_t+V'_t),
$$ and hence for $t \geq T_1$,
$$
|K_{t}-K'_{t}|\leq
 |K_{T_1}-K'_{T_1}|
 + 2|a_0| \left( \sup_{T_1\leq s\leq t} (|X_s-X'_s|+|U_s-U'_s| )+
 \int_{T_1}^t |K_s-K'_s|\, ds\right).
 $$
  By (\ref{eqf:2}), for $t\geq T_1$,
$$
|K_{t}-K'_{t}|
 \leq e^{2|a_0|( t-T_1)} \left( |K_{T_1}-K'_{T_1}| + 2|a_0|
 \sup_{T_1\leq s\leq t} (|X_s-X'_s|+|U_s-U'_s|
 )\right).
 $$
Recall that $t_1=1/k_0$.
 It follows that $T_{k_0} -T_1 \leq 1$
and
 $e^{2|a_0|(t-T_1)} \leq e^{2|a_0|}$
for $t\leq T_{k_0}$. The definition of the $T_k$'s implies that
$$
\sup_{0\leq s\leq T_k} (|X_s-X'_s|+|U_s-U'_s|) \leq
 2
\lambda^{k-1} \eps.
$$
We obtain for $t\leq T_k$ with $k\leq k_0$,
\begin{equation}\label{eqf:3}
 |K_{t}-K'_{t}|
 \leq  e^{2|a_0|} (1+4|a_0|) \lambda^{k-1} \eps.
\end{equation}

We have
 \begin{equs}
 \Xi &:=
 (X_{T_k} - X'_{T_k})^T \Lambda(X_{T_k})
 (X_{T_k} - X'_{T_k})
 - (X_{T_{k-1}} - X'_{T_{k-1}})^T \Lambda(X_{T_{k-1}})
 (X_{T_{k-1}} - X'_{T_{k-1}})\\
 &=
 \int_{T_{k-1}}^{T_k}
 d\left((X_t - X'_t)^T \Lambda(X_t)
 (X_t - X'_t)\right)\\
 &\leq
 c_3 \int_{T_{k-1}}^{T_k}
 \Lambda(X_t) (X_t - X'_t)
 \cdot
 ((\sigma(X_t) -\sigma(X'_t)) \,dB_t \\
 &
 +\
 c_4 \int_{T_{k-1}}^{T_k}
 \Lambda(X_t) (X_t - X'_t)
 \cdot
 (\bfb (X_t) - \bfb (X'_t)) \,dt \\
 &
 +\
 c_5 \int_{T_{k-1}}^{T_k}
 \Lambda(X_t) (X_t - X'_t)
 \cdot
 (K_t - K'_t) \,dt \\
 &
  +\
 c_6 \int_{T_{k-1}}^{T_k}
 \Lambda(X_t) (X_t - X'_t)
 \cdot
 \cn(X_t) \,dL_t  \\
 &
  -\
 c_7 \int_{T_{k-1}}^{T_k}
 \Lambda(X_t) (X_t - X'_t)
 \cdot
 \cn(X'_t) \,dL'_t\\
 &
  +\
 c_8
 \left|
 \int_{T_{k-1}}^{T_k}
 |X_t - X'_t|^2
 (\sigma(X_t)\,dB_t + \bfb (X_t) \,dt + K_t \,dt +
 \cn(X_t) \,dL_t)
 \right|
 .
 \end{equs}
Rewrite the second to the last term as
$$
  - c_7 \int_{T_{k-1}}^{T_k}
 \Lambda(X'_t) (X_t - X'_t) \cdot
 \cn(X'_t) \,dL'_t -  c_7 \int_{T_{k-1}}^{T_k}
 (\Lambda(X_t) -\Lambda(X'_t)) (X_t - X'_t) \cdot
 \cn(X'_t) \,dL'_t .
$$

We apply (\ref{eqf:3}) to $\int_{T_{k-1}}^{T_k} \Lambda(X_t)
(X_t-X'_t)\cdot (K_t-K'_t)\, dt$ and use the Lipschitz property of
$\Lambda$ to obtain
 \begin{equs}
 \Xi \leq
 c_3 & \int_{T_{k-1}}^{T_k}
 \Lambda(X_t) (X_t - X'_t) \cdot
 (\sigma(X_t) -\sigma(X'_t)) \,dB_t \\
 &
 +\
 c_4 \int_{T_{k-1}}^{T_k}
 \Lambda(X_t) (X_t - X'_t) \cdot
 (\bfb (X_t) - \bfb (X'_t)) \,dt \\
 &
 +\
 c_9 \int_{T_{k-1}}^{T_k}
 |\Lambda(X_t)  (X_t - X'_t)|
 \lambda^{k-1} \eps
 \,dt  \\
 &  +\
 c_{10} \int_{T_{k-1}}^{T_k}
 |X_t - X'_t|^2 \,dL_t   \\
 &
  +\
 c_{11} \int_{T_{k-1}}^{T_k}
 |X_t - X'_t|^2 \,dL'_t
  +\
 c_{12} \int_{T_{k-1}}^{T_k}
  |X_t - X'_t|^2 \,dL'_t \\
 &
  +\
 c_8  \left | \int_{T_{k-1}}^{T_k}
 |X_t - X'_t|^2
 (\sigma(X_t)\,dB_t + \bfb (X_t) \,dt + K_t \,dt +
 \cn(X_t) \,dL_t) \right| .   \label{eqf:6}
 \end{equs}
 Here, we made use of \eref{eqf:1} to bound the two terms involving integrations with respect to  the
 local times $L_t$ and $L_t'$.

It follows from
 (\ref{eqn:c1})
that there exists $c_{13} >0$
such that if
\begin{equation}\label{eqf:15}
  (X_{T_k} - X'_{T_k})^T \Lambda(X_{T_k})
 (X_{T_k} - X'_{T_k})
 - (X_{T_{k-1}} - X'_{T_{k-1}})^T \Lambda(X_{T_{k-1}})
 (X_{T_{k-1}} - X'_{T_{k-1}}) \leq c_{13} \lambda^{2(k-1)} \eps^2
\end{equation}
and
\begin{equation}\label{eqf:16}
 |X_{T_{k-1}} - X'_{T_{k-1}}| \leq \lambda^{k-2} \eps,
\end{equation}
then
\begin{equation}\label{eqf:14}
 |X_{T_k} - X'_{T_k}|
\leq   \lambda^{-1/3} \, \lambda^{k-1} \eps < \lambda^{k-1} \eps .
\end{equation}

In view of our assumptions on $\Lambda$ and $\sigma$,
 \begin{equs}
\var\Bigl( c_3 \int_{T_{k-1}}^{T_k}
 &\Lambda(X_t) (X_t - X'_t) \cdot
 (\sigma(X_t) -\sigma(X'_t)) \,dB_t
\ \Big| \ \FF_{T_{k-1}} \Bigr) \\
 &\leq
\E \Bigl( c_{14} \int_{T_{k-1}}^{T_k}
 \sup_{T_{k-1}\leq s \leq T_k}
 |X_s - X'_s|^4 \,dt  \ \Big| \ \F_{T_{k-1}} \Bigr)\\
 &\leq c_{14} \lambda^{4(k-1)} \eps^4 t_1.
 \end{equs}
We make $t_1>0$ smaller (and therefore $k_0$ larger), if
necessary,
 so that by Doob's and Chebyshev's inequalities,
\begin{equation}\label{eqf:7}
 \P \left(
 \Big| c_3 \int_{T_{k-1}}^{T_k}
 \Lambda(X_t) (X_t - X'_t) \cdot
 (\sigma(X_t) -\sigma(X'_t)) \,dB_t\Big|
 \geq (1/10)c_{13} \lambda^{2(k-1)} \eps^2
 \ \Big| \  \FF_{T_{k-1}} \right) \leq
 1/8.
\end{equation}
We make $t_1>0$ smaller, if necessary, so that
\begin{eqnarray}
\left|  c_4 \int_{T_{k-1}}^{T_k}
 \Lambda(X_t) (X_t - X'_t) \cdot
 (\bfb (X_t) - \bfb (X'_t)) \,dt \right|
 &\leq &
 c_{15} \int_{T_{k-1}}^{T_k} |X_t - X'_t|^2 \,dt \label{eqf:8}\\
 &\leq & c_{15} \int_{T_{k-1}}^{T_k}
 \sup_{T_{k-1}\leq s \leq T_k}
 |X_s - X'_s|^2 \,dt \nonumber \\
 &\leq & c_{15} t_1 \lambda^{2(k-1)} \eps^2
 \leq (1/10)c_{13} \lambda^{2(k-1)} \eps^2. \nonumber
\end{eqnarray}
If $t_1>0$ is sufficiently small then
 \begin{eqnarray}
\left|  c_9 \int_{T_{k-1}}^{T_k}
 |\Lambda(X_t) (X_t - X'_t)|
 \lambda^{k-1} \eps \,dt \right|
 &\leq&
 c_{17} \int_{T_{k-1}}^{T_k}
 \sup_{T_{k-1}\leq s \leq T_k}
 |X_s - X'_s|
 \lambda^{k-1} \eps
 \,dt \label{eqf:11}  \\
 &\leq &
 (1/10)c_{13} \lambda^{2(k-1)} \eps^2. \nonumber
 \end{eqnarray}

We make $c_0>0$ in the definition of $T_k$ so small that
\begin{align}
 c_{10} \int_{T_{k-1}}^{T_k}
 |X_t - X'_t|^2 \,dL_t & +
 c_{11} \int_{T_{k-1}}^{T_k}
 |X_t - X'_t|^2 \,dL'_t +
 c_{12} \int_{T_{k-1}}^{T_k}
  |X_t - X'_t|^2 \,dL'_t\label{eqf:12}\\
 &\leq
 (1/10)c_{13} \lambda^{2(k-1)} \eps^2.\nonumber
\end{align}
Completely analogous estimates show that for sufficiently small
$t_1, c_0>0$,
\begin{equs}
 \P \Bigg(
 \Big| c_8 \int_{T_{k-1}}^{T_k}
 |X_t - X'_t|^2
& (\sigma(X_t)\,dB_t + \bfb (X_t) \,dt + K_t \,dt
 + \cn(X_t) \,dL_t )
 \Big|
 \\ &
 \geq (1/10)c_{13} \lambda^{2(k-1)} \eps^2
 \ \Big|\ \FF_{T_{k-1}} \Bigg) \leq 1/8.\label{eqf:13}
\end{equs}

Combining
(\ref{eqf:6}) with (\ref{eqf:7})-(\ref{eqf:13}),
  we see that conditioned on
$\FF_{T_{k-1}}$,  with probability 3/4 or greater the following
event holds:
 $$
  (X_{T_k} - X'_{T_k})^T \Lambda(X_{T_k})
 (X_{T_k} - X'_{T_k})
 - (X_{T_{k-1}} - X'_{T_{k-1}})^T \Lambda(X_{T_{k-1}})
 (X_{T_{k-1}} - X'_{T_{k-1}}) \leq c_{13} \lambda^{2(k-1)} \eps^2.
 $$
In view
of
 (\ref{eqf:15})-(\ref{eqf:14}), this implies that if
(\ref{eqf:16}) holds, then
\begin{equation}\label{eqf:9n}
|X_{T_k} - X'_{T_k}| \leq
 \lambda^{k-1} \eps.
\end{equation}

Let
$$
 T'_k = (T_{k-1} + t_1) \land
 \inf\{t>T_{k-1}: |X_t - X'_t| \geq
 \lambda^{k-1} \eps\}.
$$
By Doob's inequality, if $t_1>0$ is sufficiently small,
\begin{equation}\label{eqf:9}
 \P \left(
 \sup_{T_{k-1}\leq s \leq T'_k}
 |U_s - U'_s| \geq (1/2) \lambda^{k-1} \eps
 \ \Big| \ \FF_{T_{k-1}} \right) \leq 1/4.
\end{equation}

Let
 $$
F_k =\{ T_k=T_{k-1} +t_1\}\cup \{ |L_{T_k} - L_{T_{k-1}}| \lor
|L'_{T_k} - L'_{T_{k-1}}|> c_0\}.
 $$
By (\ref{eqf:9n}) and (\ref{eqf:9}),
 \begin{equation}\label{eqf:17}
 \P( F_k \mid \FF_{T_{k-1}} ) \geq 1/2 .
\end{equation}

Recall the definition of $k_0$ relative to $p_1\in (0, 1)$.
Repeated application of the strong Markov property at the
stopping times $T_k$ and comparison with a Bernoulli sequence
prove
 that at least $k_0/4$ of the events $F_k$ will occur with
probability $p_1$ or greater. This implies that with
probability $p_1$ or greater,
 $$ \hbox{either} \quad  T_{k_0} -T_1 \geq 1/4 \quad \hbox{or} \quad
 L_{T_{k_0}} - L_{T_1} \geq c_0/(4t_1) \quad \hbox{or} \quad
 L'_{T_{k_0}} - L'_{T_1} \geq c_0/(4t_1). $$
From the definitions of
$T_1$
 and $S$,
 $$\P \left( \sup_{t\in [0, S]}|X_t - X'_t| \leq \eps \lambda^{k_0} \right)\geq p_1.
 $$
 As $\eps>0$ is arbitrary, we have
 $$  \P \left(  X_t = X'_t \hbox{ for every } t\in [0, S] \right)\geq
 p_1>0,
 $$
which contradicts
 the  definition of
$T_1$.
 This completes the proof of pathwise uniqueness.

\medskip

Strong existence follows from weak existence and pathwise
uniqueness using a standard argument; see \cite[Section
IX.1]{RY}.
\end{proof}

  \section{Diffusions with gradient drifts}
\label{sect:potential}

This section is devoted to analysis of a diffusion with inert
drift given as the gradient of a potential. We will use such
diffusions to approximate diffusions with reflection. The analysis
of the stationary measure is easier in the case when the drift is
smooth. Let $V$ be a $C^1$ function on $D$ that goes to $+\infty$
sufficiently fast as $x$ approaches the boundary of $D$. We will
consider the diffusion process on $D$ associated with generator
${\cal L}=\frac 12 e^{V} \nabla (e^{-V}A\nabla)$ but with an
additional  ``inert" drift. More precisely, let $\Gamma$ be a
non-degenerate constant $d\times d$-matrix. We consider the SDE
\begin{equs}\label{e:diffV1L}
\begin{cases}
dX_t &= \sigma (X_t) \,dB_t +{\bf b}(X_t) \,dt -\frac12
(A\nabla V)(X_t)\, \,dt + K_t\, \,dt \;, \\
dK_t &= -\frac12 \Gamma \, \nabla V (X_t)\, dt\;,
 \end{cases}
\end{equs}
where $A=A(x) = (a_{ij}) = \sigma^T\sigma$ is uniformly elliptic
and bounded. We assume that $\sigma \in C^1 (\overline D)$ and
${\bf b}=(b_1, \cdots, b_d)$ with $b_k(x):=\frac12 e^{V(x)}
\sum_{i=1}^d \partial_i \left( e^{-V(x)} a_{ik}(x)\right)$. Note
that, since $\sigma \in C^1(\overline D)$, $V\in C^1(D)$ and ${\bf
b}$ is bounded, \eqref{e:diffV1L} has a unique weak solution $(X,
K)$ up to the time $\inf\{t>0: X_t\notin D \hbox{ or }
K_t=\infty\}$.

To find a candidate for the stationary distribution for $(X,
K)$, we will do some computations with processes of the form
$f(X_t, K_t)$, where $f\in C^2(\RR^{2d})$. We will use the
following notation,
\begin{eqnarray*}
 \nabla_x f(x,y) &=& \left(\frac{\partial}{\partial x_1}
 f(x_1, \dots, x_d, y_1, \dots,
y_d), \dots, \frac{\partial}{\partial x_d} f(x_1, \dots, x_d, y_1,
\dots, y_d)\right)^T, \\
 \nabla_y f(x,y) &=&  \left(\frac{\partial}{\partial y_1}
 f(x_1, \dots, x_d, y_1, \dots,
y_d), \dots, \frac{\partial} {\partial y_d} f(x_1, \dots, x_d,
y_1, \dots,
y_d)\right)^T, \\
\LL_x f(x,y) &=& \frac12 \, e^{V(x)}\, \sum_{i,j=1}^d
\frac{\partial}{\partial x_i} \left(e^{-V(x)} a_{ij} (x)
 \frac{\partial}{\partial x_j} f (x,y) \right), \\
\LL^*_x f(x,y) &=& \frac12 \,  \sum_{i,j=1}^d
\frac{\partial}{\partial x_i} \left(e^{-V(x)} a_{ij} (x)
\frac{\partial}{\partial x_j} (e^{V(x)} f (x,y) \right).
\end{eqnarray*}
For $f\in C^2(\RR^{2d})$, by Ito's formula, we have
\begin{eqnarray*}
df(X_t, K_t) &=& \nabla_x f dX_t+ \nabla_y f \,dK_t+\frac12
\sum_{i,j=1}^d \partial_i \partial_j f d\<X^i, X^j\>_t\\
&=& \hbox{local martingale}+ \left( \LL_x f   + \nabla_x f\cdot
K_t -\frac12 \nabla_y f \cdot \Gamma \nabla_x V \right) \,dt .
\end{eqnarray*}
So the process $(X, K)$ has the generator
\begin{equation}\label{generator}
 \GG f(x, y):=\LL_x f(x, y)+y\cdot
\nabla_x f(x, y) -\frac12
  \Gamma \, \nabla_x V(x)\cdot \nabla_y f(x, y)
\end{equation}
 for $ x\in D$   and   $y\in \RR^d$.

\medskip

We will now assume that $(X,K)$ has a stationary measure of a
special form. Then we will do some calculations to find an
explicit formula for the stationary measure, and finally we will
show that the calculations can be traced back to complete the
proof that the measure we started with is indeed the stationary
distribution.

Suppose that $(X, K)$ has a stationary distribution $\pi$ of the
form $\pi (dx, dy)=\rho_1(x) \rho_2(y)\,dx\, dy$. It follows that
${\cal G}^* \pi =0$, in the sense that for every $f\in {\rm
Dom}({\cal G})$,
$$ \int_{\RR^{2d}} {\cal G} f(x, y) \pi (dx, dy) =0.$$
Then we have for every $f\in C^2_c(D\times \RR^d)$, by the
integration by parts formula,
\begin{eqnarray*}
0 &=&\int_{\RR^d} \left( \int_D \left( \rho_1(x) y\cdot \nabla_x f(x,
y) + \rho_1(x) \LL_x f(x, y) \right) \,dx \right) \rho_2(y) \,dy\\
 &&
 -\frac12 \int_{D} \left( \int_{\RR^d} \rho_2(y) \Gamma  \, \nabla_x V(x)
\cdot
  \nabla_yf(x, y)   \,dy \right) \rho_1(x) \,dx \\
&=& \int_{\RR^d} \left( \int_D \left(-\nabla_x \rho_1(x) \cdot y f(x, y)
 +
  \LL_x^* \rho_1 (x) f(x, y)\right) \,dx \right) \rho_2(y) \,dy \\
&& + \frac 12
 \int_{D} \left( \int_{\RR^d}  \Gamma \, \nabla_x V(x) \cdot
\nabla_y \rho_2 (y) f(x, y) \,dy
      \right) \rho_1(x) \,dx \;.
\end{eqnarray*}
This implies that
\begin{equation}\label{eqn:2}
\rho_2(y)\left( \nabla_x \rho_1 (x) \cdot y
  - \LL_x^*
 \rho_1(x)\right) -\frac12
 \rho_1(x)  \Gamma \, \nabla_x V(x) \cdot
\nabla_y \rho_2 (y) =0 \qquad  \hbox{for } (x, y)\in D \times
\RR^d .
\end{equation}
We now make an extra assumption that $\rho_1(x)= ce^{-V(x)}$.
Then we have
 $$ \rho_2 (y) \nabla_x V(x) \cdot y + \frac12
 \Gamma \, \nabla_x V(x) \cdot \nabla_y \rho_2 (y)=0
 \qquad \hbox{for } (x, y)\in D\times \RR^d.
 $$

Since $V(x)$ blows up as $x$ approaches the boundary $\partial D$,
$\{\nabla_x V (x), x\in D\}$ spans the whole of $\RR^d$. (If not,
there exists $v \in \RR^d$ such that $\langle v, \nabla V(x)
\rangle \le 0$ for every $x \in \RR^d$. One then gets a
contradiction to the fact that there exists $r_0 > 0$ (possibly
$r_0 = +\infty$) such that $\lim_{r \to r_0} V(x + vr) = \infty$
for every $x \in D$.) So we must have
$$ y + \frac12
 \Gamma ^T \nabla_y \log \rho_2 (y) = 0  \qquad \hbox{for
every } y\in \RR^d.
$$
This implies that $\Gamma$ is symmetric and that
$$ \nabla_y \log \rho_2 (y) =
-2 \Gamma^{-1} y  \qquad \hbox{for every } y\in \RR^d.
$$
Hence we have
$$ \log \rho_2 (y) = -
(\Gamma^{-1}  y, y) +c_1,
$$
or
$$ \rho_2 (y)= c_2 \exp \Big( -
 (\Gamma^{-1}  y, y) \Big).
$$

The above calculations suggests that when $\Gamma$ is a symmetric
positive definite matrix, the stationary distribution for the
process $(X, K)$ in (\ref{eqn:1}) has the form
$$ c_3 \, {\bf 1}_D(x) \,  \exp \Big( -V(x) -
(\Gamma^{-1}  y, y) \Big)
\, dx\, \,dy,
$$
where $c_3>0$ is the normalizing constant. This is made rigorous
in the following result. Recall that a process is said to be
conservative if it does not explode in finite time. In the present setting, since
we consider processes taking values on $D\times \RR^d$, the is means that the
second component does not explode and that the first component does not reach the boundary of $D$.

\begin{thm}\label{T:1} Suppose that $\Gamma$ is a symmetric positive
definite matrix, each $\sigma_{ij}$ is $C^1$ on $\overline D$ so
that $A=\sigma^T \sigma$ is uniformly elliptic and bounded on $D$,
and $V$ is a $C^1$ potential on $D$. Suppose that
$$
\pi (dx, dy):=c_0 \1_D(x) \, \exp \Big( -V(x)
  - (\Gamma^{-1}  y, y) \Big) \, dx\, dy  $$
is a probability measure on $D\times \RR^d$ such that the
diffusion process of \eqref{e:diffV1L} with initial distribution
$\pi$ is conservative. Then the process of \eqref{e:diffV1L} has $\pi$ as a
(possibly not unique) stationary distribution.
\end{thm}

\pf Let $\GG$ be the operator defined by \eqref{generator} with domain
$\DD (\GG)=C^2_c(D\times \R^d)$. The above calculation shows that
\begin{equation}\label{eqn:G}
 \int_{D\times \RR^d} \GG f(x, y) \pi (dx, dy)=0 \qquad \hbox{for every }
 f\in \DD (\GG).
 \end{equation}

Let $E:=D\times \R^d$ and $E_\Delta = E\cup \{\Delta\}$ be the
one-point compactification of $E$. Recall $\GG$ is defined by
\eqref{generator} with  $\DD (\GG)=C^2_c(E)$. Since \eqref{eqn:G}
holds, we have by Theorem 9.17 in Chapter 4 of \cite{EK} that
$\pi$ is a stationary measure for some solution $\P$ to the
martingale problem for $(\GG, \DD (\GG))$. In \cite[Theorem
4.9.17]{EK}, the measure $\P$ is a probability measure on the
product space $E^{\RR_+}$. However by \cite[Corollary 4.3.7]{EK},
any solution of the martingale problem for $(\GG, \DD (\GG))$ has
a modification with sample paths in the Skorokhod space $\bD
(\RR_+, E_\Delta)$ of right continuous paths on $\E_\Delta$ having
left limits. Thus we can assume that $\P$ is a $\bD (\RR_+,
E_\Delta )$-solution to the martingale problem $(\GG, \DD(\GG))$
with stationary distribution $\pi$. Let $(X,   K)$ denote the
coordinate maps on
 $E_\Delta$ and set
 $$\zeta:=\inf\{t>0: (X_t, K_t)=\Delta\}.
 $$
We show next that $\{(X_t, K_t), t<\zeta\}$ under $\P$ is a
solution to SDE  \eqref{e:diffV1L} up to
time $\zeta$.

Let $f(x, y)= g(x)h(y)$ with $g\in C_c^2(D)$ and $h\in
C_c^2(\R^d)$. Then under $\P$,
\begin{align}
 g(X_t)h(K_t)&= g(X_0)
 h(K_0)  + \hbox{martingale} \label{eqn:gh} \\
& + \int_0^t \Big( h(K_s) \LL_x f(X_s) + h(K_s) K_s \cdot \nabla_x
g(X_s)\nonumber\\ &\qquad - \frac 12 g(X_s) \Gamma \nabla_x V(X_s) \cdot \nabla_y
h(K_s) \Big) \,ds. \nonumber
\end{align}
Here and below, we use the convention that for any function $f$ on
$E$, $f(\Delta):=0$. Since $\pi$ is the stationary distribution
for $(X, K)$, by the dominated convergence theorem, the above
holds for $g\in C^2(\overline D)$ and   $h\in C^2(\RR^d)$ that
have  bounded first derivatives. So in particular, we have for
$g\in C^2(\overline D)$, under $\P$
$$
 g(X_t)= g(X_0)  + \hbox{martingale} + \int_0^t \left( \LL_x g (X_s)
+  K_s \cdot \nabla_x g(X_s)  \right) \,ds \qquad \hbox{for } t\geq
0.
$$
Let $g(x)=x_i$. There are local martingales $M=(M^1, \cdots, M^d)$ such that
$$
dX_t= dM_t + {\bf b} (X_t) \,dt -\frac12 (A\nabla V )(X_t) \,dt + K_t
\,dt, \qquad t<\zeta.
$$
Applying Ito's formula to $X^i_t X^j_t$ yields $\<M^i, M^j\>_t=\<
X^i, X^j\>_t = \int_0^t a_{ij}(X_s) \,ds$ for $t<\zeta$. Define
$dB_t:= \sigma^{-1} (X_t)\, dM_t$. Then $\{B_t, t<\zeta\}$ is a
Brownian motion up to time $\zeta$ and $M_t=\int_0^t \sigma (X_s)
\,dB_s$ for $t<\zeta$. Hence we have
\begin{equation}\label{eqn:G2}
dX_t= \sigma (X_t) \,dB_t + {\bf b} (X_t) \,dt -\frac12 (A\nabla V
)(X_t) \,dt + K_t \,dt, \qquad t<\zeta.
 \end{equation}
From \eqref{eqn:gh}, we have for every $h\in C^2 (\RR^d)$ that has
bounded derivatives,
$$   h(K_t)=
h(K_0) + \hbox{local martingale}  -\frac12  \int_0^t \Gamma
\nabla_x V(X_s) \cdot \nabla_y h(K_s)   \,ds, \qquad t<\zeta.
$$
In particular, taking $h(y)=y_i$, $1\leq i\leq d$, there are local
martingales $N=(N^1, \cdots, N^d)$ such that
$$ dK_t = dN_t -\frac12 \Gamma \nabla_x V(X_s) \,ds, \qquad t<\zeta.
$$
Applying Ito's formula to $K^iK^j$ yields $\<N^i, N^j\>_t=\<K^i,
K^j\>_t=0$ for $t<\zeta$. Hence
$$ dK_t =  -\frac12 \Gamma \nabla_x V(X_s) \,ds, \qquad t<\zeta.
$$
This together with \eqref{eqn:G} implies that $\{(X_t, K_t),
t<\zeta\}$ under $\P$ is a (continuous) weak solution to
(\ref{e:diffV1L}) with
  initial distribution $\pi$
up to time $\zeta$. Since $a_{ij}, V\in C^1(\overline D)$, weak
uniqueness holds for solutions of \eqref{e:diffV1L} (see \cite{SV}).
So under our conservativeness assumption, \eqref{e:diffV1L} has a
conservative weak solution with initial distribution $\pi$ which is
unique in distribution. By standard techniques (cf.~the proof of
\cite[Proposition I.2.1]{Bass97}), any weak solution to
\eqref{e:diffV1L} gives rise to a solution to the martingale problem
for $(\GG, \DD (\GG ))$. This implies that $\zeta =\infty$ because
$(X,K)$ with the initial distribution $\pi$ is a conservative
solution to \eqref{e:diffV1L}, by assumption. We conclude that $\pi$
is a stationary distribution for \eqref{e:diffV1L}.
 \qed

\bigskip

We will present an easily verifiable condition on $V$ for
Theorem \ref{T:1} to be applicable. The following preliminary
result for diffusions without inert drift may have interest of
its own.

Recall that $W^{1,2}_0(D)$ is the closure of the space
$C^\infty_c(D)$ of smooth functions with compact support in $D$
under the Sobolev norm $\| u\|_{1,2} := \left(\int_D (|u(x)| +
\sum_{i=1}^d |\partial_i u(x)|^2 )\,dx\right)^{1/2}$.

\begin{theorem}\label{T:4.2} Let $D\subset \R^d$ be a domain
(i.e. a connected open set) and $A(x)=(a_{ij}(x))$ be a
measurable $d\times d$ matrix-valued function on $D$  that is
uniformly elliptic and bounded. Suppose that  $\varphi$ is a
function in $W^{1,2}_0(D)$ that is positive on $D$. Then the
minimal diffusion $X$ on $D$ having infinitesimal generator ${\cal
L} = \frac1{ 2\varphi^2}  \sum_{i, j=1}^d
\partial_i (\varphi^2 a_{ij} \partial_j)$ is conservative.
\end{theorem}

\proof The process $X$ is the symmetric diffusion (with respect to
the symmetrizing measure $\varphi (x)^2 \,dx$) associated with the
Dirichet form $(\EE, \FF)$ in $L^2(D, \varphi (x)^2 \,dx)$, where
$$ \EE (u, v) = \frac12 \int_D \sum_{i,j=1}^d a_{ij}(x) \partial_i u(x) \partial_j v(x)
\, \varphi (x)^2 \,dx ,
$$
and $\FF$ is the closure of $C_c^\infty (D)$ with respect to the
norm $\sqrt{\EE_1}$; here we define
$$\EE_1 (u, u):=\EE (u, u) +
\int_D u(x)^2 \varphi (x)^2 \,dx. $$ Let $X^0$ be the symmetric
diffusion in $D$ with respect to Lebesgue measure on $D$
associated with the Dirichlet form $(\EE^0, W^{1,2}_0(D))$ in
$L^2(D, dx)$, where
$$ \EE^0 (u, v) = \frac12 \int_D \sum_{i,j=1}^d a_{ij}(x)
\partial_i u(x) \partial_j v(x)
\,   dx .
$$
Then by \cite[Theorems 2.6 and 2.8]{CFTYZ}, $X$ can be obtained
from $X^0$ through a Girsanov transform using the  martingale
$dZ_t= \varphi (X_t^0)^{-1} dM^{\varphi}_t$, where $M^\varphi$  is
the martingale part of the Fukushima decomposition for $\varphi
(X_t^0)-\varphi (X_0^0)$. We now conclude from \cite[Theorem
2.6(ii)]{CFTYZ} that $X$ is conservative. \qed

\bigskip

\begin{theorem}\label{T:4.3}
Suppose that $\Gamma$ is a constant symmetric positive definite
matrix. Let $D\subset \RR^d$ be a bounded domain, $\sigma =
\sigma(x)$ be a $d\times d$ matrix that is $C^1$ on $\overline D$
so that $A=\sigma^T \sigma$ is uniformly elliptic and bounded on
$D$. Suppose that $V$  is a $C^1$ function in $D$ such that
$\varphi=e^{-V/2}\in  W^{1,2}_0 (D)$. Then  the (minimal) solution
of \eqref{e:diffV1L} with initial distribution $\pi$ is
conservative and has $\pi$ as its stationary distribution.
\end{theorem}

\pf  Let $(X, \P_x)$ denote the symmetric diffusion process with
infinitesimal generator
$$
\LL_x:=\frac12 e^V \sum_{i, j=1}^d \partial_i
\left(e^{-V} a_{ij} \partial_j \right).
$$
By Theorem \ref{T:4.2}, $(X, \P_x)$ is conservative.
Let $K_t=K_0-\frac12  \int_0^t (\Gamma \nabla) (X_s) \,ds$. Define for $t\geq 0$,
$$ M_t:=\exp \left( \sigma^{-1} (X_s) K_s \,dB_s -\frac12 \int_0^t | \sigma^{-1}
(X_s) K_s |^2 \,ds \right) .
$$
Clearly $\{M_t, t\geq 0\}$ is a continuous positive $\P_x$-local
martingale for every $x\in D$ with respect to the minimal
augmented filtration $\{\FF_t, t\geq 0\}$ of $X$. For $n\geq 1$,
let $T_n:=\inf\{t>0: |K_t|\geq 2^n\}$. Since $X$ is conservative,
we have $T_\infty:=\lim_{n\to \infty} T_n=\infty$, $\P_x$-a.s.,
for all $x\in D$.

Let $\P$ be the distribution of the symmetric stationary diffusion
$X$ with initial probability distribution $c_1 e^{-V}\, dx$ in $D$
and $K\equiv 0$. The existence of such a measure follows, for
example, from Theorem \ref{T:4.2} with $\Gamma \equiv 0$. Let
$K_0$ have the Gaussian distribution $c_2 \exp (-(\Gamma^{-1}y,
y))$. Define a new probability measure $\Q$ on $\FF_\infty$ by
$$ \frac{d\Q }{d\P } = M_t \qquad \hbox{on } \FF_t  \quad \hbox{for every } t\geq
0.
$$
By the Girsanov theorem, the process
$$ W_t:=B_t-\int_0^t \sigma^{-1}(X_s) K_s \,ds$$
is a Brownian motion under $\Q $ up to the time $T_\infty$, so
under $\Q$,
\begin{equation}\label{e:xk}
 \begin{cases}
 d X_t = \sigma (X_t) \,dW_t +{\bf b}(X_t) \,dt -\frac 12 (A\nabla V)(X_t) \,dt
    + K_t \,dt , \\
  dK_t = -\frac12 \Gamma \nabla V(X_t) \,dt,
  \end{cases}
  \qquad \hbox{for } 0\leq t<T_\infty.
  \end{equation}
Note that $(X, K)$ has initial distribution $\pi$.

Conversely, given a solution $(X, K)$ of \eqref{e:xk} under the
measure $\Q$, the process $X$ is a conservative $\LL_x$-diffusion
under the measure
$$\exp \left( -\sigma^{-1} (X_s) K_s \,dW_s -\frac12 \int_0^t | \sigma^{-1}
(X_s) K_s |^2 \,ds \right)d\Q.
$$
So to prove that the solution to \eqref{e:diffV1L} is
conservative, it suffices to show that $\Q (T_\infty =\infty )
=1$.

We are going to show that $\{M_t, t\geq 0\}$ is in fact a
positive $\P$-martingale. This implies that  $\Q (T_\infty=
\infty )=1$ because $\Q(T_\infty >t)=\E_\P \left[ M_t \right]
=1$ for every $t>0$.

Let $E:=D\times \R^d$ and $E_\Delta = E \cup \{\Delta\}$ be the
one-point compactification of $E$. Recall $\GG$ is defined by
\eqref{generator} with  $\DD (\GG)=C^2_c(E)$. Since \eqref{eqn:G}
holds, by the same argument as in the proof of Theorem \ref{T:1},
we deduce that   $\pi$ is a stationary measure
for some solution $\wh \Q$ on $\bD (\RR_+, E_\Delta)$ to the
martingale problem for $(\GG, \DD (\GG))$. Let $(\wh X, \wh K)$
denote the coordinate maps on $E_\Delta$ and set
$\zeta:=\inf\{t>0: (\wh X_t, \wh K_t)=\Delta\}$. Then   $\{(\wh
X_t, \wh K_t), t<\zeta \}$ satisfies the SDE \eqref{e:xk}, and
consequently, it has the same distribution as $\{(X_t, K_t),
t<T_\infty \}$ under $\Q$.

Note that the matrix $\sigma^{-1}$ is bounded so there is a constant
$c_1>0$ such that
$$|\sigma^{-1}(x) v| \leq c_1 |v| \qquad \hbox{for every } x\in
\overline D \hbox{ and } v\in \RR^d.
$$
Since under $\wh \Q$,  $\wh K_t$ has the same Gaussian
distribution for every $t\geq 0$, there exist $c_2$ and $c_3$ such
that if $r\leq c_2$, then
$$ \E_{\wh \Q} \left[ \exp \left( r \left| \sigma^{-1} (\wh X_s) \wh K_s \right|^2
  \right)\right]
\leq  \E_{\wh \Q} \left[ \exp \left( r c_1^2  \, \big| \wh K_s
\big|^2 \right)\right]\leq c_3 \qquad \hbox{for every } s\geq 0.
$$
By Jensen's inequality applied with the measure $\frac{1}{t_0}
1_{[0,t_0]}(s)\, ds$ with $t_0\in (0, c_2]$ and the function $e^{x}$
we have
\begin{align*}
\E_{\wh \Q} \left[ \exp \left( \int_0^{t_0} \left| \sigma^{-1} (\wh
X_s) \wh K_s \right|^2\, ds\right) \right] &=\E_{\wh \Q} \left[ \exp
\left( \frac{1}{t_0}\int_0^{t_0} \Big(t_0 \left| \sigma^{-1} (\wh X_s)
\wh K_s \right|^2\Big)
\,ds \right) \right] \\
&\leq \E_{\wh \Q} \left[ \frac{1}{t_0}\int_0^{t_0} \exp \left( t_0
\left|
 \sigma^{-1} (\wh X_s) \wh K_s \right|^2 \right)
\, ds \right]\\
&\leq \frac1{t_0} \int_0^{t_0} c_3\, ds =c_3 .
\end{align*}

Define $N_t=\int_0^t \sigma^{-1}(X_s) K_s\, dB_s$. This is a
martingale with respect to $\P$ and its quadratic variation $\< N
\>_t$ is equal to  $\int_0^t | \sigma^{-1}(X_s) K_s|^2\, ds$ under
both $\P$ and $\Q$. Note that $\exp(-N_t-\frac12 \angel{N }_t)$ is
a positive local martingale with respect to $\P$ and hence a
$\P$-supermartingale. Recall that $T_\infty$ is the lifetime for
$(X, K)$, i.e., $(X_t, K)=\Delta$ for $t\geq T_\infty$, and that
by convention, every function $f$ on $E$
is extended to a function on $E_\Delta$ by setting $f(\Delta)=0$.
In particular, we have $N_t=N_{t\wedge T_\infty}$. We have
$$\E_{\Q} \Big[e^{-2N_t}\Big]=\E_\P\Big[ e^{-N_t-\frac12 \angel{N}_t}\Big]\leq 1.
$$
Using  Cauchy-Schwartz,
\begin{align*}
\E_\P \left[ \exp \left( \int_0^t \left| \sigma^{-1}(X_s) K_s
\right|^2\, ds \right) \right]
&=\E_\P \Big[e^{\angel{N }_t}\Big]\\
&=\E_{\Q } \Big[ e^{\angel{N }_t} e^{-N_t-\frac12 \angel{N}_t}  \Big]\\
&\leq \Big( \E_{\Q } \left[ e^{-2N _t} \right]\Big)^{1/2}
\Big(\E_{\Q } \left[ e^{\angel{N }_t} \right]\Big)^{1/2}\\
&\leq  \left( \E_{\Q } \left[ \exp \left( \int_0^{t }
\left| \sigma^{-1}(X_s) K_s \right|^2 \,ds \right)
\right] \right)^{1/2} \\
&\leq \left(\E_{\wh \Q } \left[ \exp \left( \int_0^t \left|
\sigma^{-1}(\wh X_s) \wh K_s \right|^2 \,ds \right)
\right]\right)^{1/2} .
\end{align*}
As we observed in the previous paragraph, the last term is bounded
if $t\leq c_2$. It follows from Novikov's criterion (see
\cite[Proposition VIII.1.15]{RY}) that $\{M_t, t\in [0, c_2]\}$ is a
uniformly integrable $\P$-martingale. It follows that $\Q (T_\infty
>c_2)=1$. Using the Markov property, we have $\Q(T_\infty =\infty
)=1$. Consequently, $\pi$ is a stationary distribution for $(X,
K)$ of \eqref{e:diffV1L}. This proves the theorem.
 \qed

\bigskip

The next corollary follows immediately from Theorem \ref{T:4.3} and
the fact that the solution of \eqref{e:diffV1L} depends in a
continuous way in its initial starting point $(x_0, k_0)$.

\begin{corollary} \label{C:4.4}
Under the conditions of Theorem \ref{T:4.3}, for every $(x_0, k_0)\in
D\times \RR^d$, the minimal solution of \eqref{e:diffV1L} with
initial value $(x_0, k_0)$ is conservative.
\end{corollary}

The following remark gives some sufficient conditions for $\varphi
\in W^{1,2}_0(D)$ and thus for the condition of Theorem \ref{T:1}
to hold with $e^{-V}=\varphi^2$.

\begin{remark}\label{R:4.5}
\begin{description} {\rm
\item{(i)} Let $ W^{1,2}(D):=\{f\in L^2(D, dx): \, \nabla f \in L^2(D, dx)\}$.
It is known (see, for example, \cite{FOT}) that  a function  $u\in W^{1,2}(D)$
is in $W^{1,2}_0(D)$ if and only if its quasi-continuous version vanishes quasi-everywhere
on $\partial D$.
So in particular, if $u\in W^{1,2}(D)$ and if
$$\lim_{x\in D\atop x\to z} u(x)=0 \ \hbox{ for every } z\in \partial D,
$$
then $u\in W^{1,2}_0(D)$.

\item{(ii)} For any positive bounded function $u\in
W^{1,2}_0(D)\cap C^1(D)$, $\varphi=e^{-1/u}\in W^{1,2}_0(D)\cap
C^1(D)$ and so Theorem \ref{T:1} is applicable for $V=2/u$ in view
of Theorem \ref{T:4.3}. This follows because $f:= \exp (-1/(2x))$
is bounded and Lipschitz continuous on $[0, k]$ for every $k\geq
1$. So by the normal contraction property (cf. \cite{FOT}),
$V=f(u)\in W^{1,2 }_0(D)$. In particular, this is the case if $u$
is the first positive eigenfunction for the Dirichlet Laplacian on
$D$ (if such a ground state exists).

\item{(iii)} The assumption that $e^{-V/2}\in C^1(D)\cap
W^{1,2}_0(D)$ is close to being optimal for ensuring that the
solution to \eref{e:diffV1L} never hits the boundary of the domain
$D$. Consider the case $n = 1$, $D = \RR_+$, $\sigma =1$, $V
\colon \RR_+ \to \RR_+$, and remove the inert drift so that $dX =
-\frac12 \nabla V(X)\,dt + dB$. Define a function $\Psi \colon
\RR_+ \to \RR$ by
\begin{equ}
\Psi(x) = \int_1^x \exp \bigl( V(y)\bigr)\, dy\;.
\end{equ}
Then one can check that $\Psi(X_t)$ is a local martingale with
quadratic variation greater or equal to $1$. In particular, it has
a positive probability of taking arbitrarily large values during
any finite time interval. If $V$ now diverges at $0$ sufficiently
slowly so that $\exp( V)$ is still integrable, this easily implies
that this diffusion has a positive probability of reaching $0$ in
any finite time.

Let us now consider the case where $\exp\bigl( V(x)\bigr) =
C/x^\alpha$ near $x = 0$. In this case, $\frac{d}{dx}e^{-V/2}$  is
$L^2$-locally integrable near $0$ if and only if $\alpha > 1$, which
is almost the range of parameters for
which $\exp(V)$ is no longer integrable at $0$, namely, $\alpha \geq 1$. }
\end{description}
\end{remark}

\bigskip

The remainder of this section is devoted to the analysis of the
case when both $\sigma $ (and hence $A$) and $\Gamma$  are the
identity matrix. We consider the SDE
\begin{equs}\label{e:diffV}
\begin{cases}
dX_t = - \frac12 \nabla V(X_t)\, dt + K_t\, dt + \,dB_t\;, \\
dK_t = - \frac12 \nabla V(X_t)\, dt \; .
\end{cases}
\end{equs}

It is possible to show under very weak additional assumptions that
$\pi$ is the only invariant measure for \eref{e:diffV}. This is
not obvious for example in the case where $V(x) = f(d(x, \partial D))$ for
some function $f\colon \RR_+ \to \RR_+$ which has a singularity at
$0$ and is such that $f(r) = 0$ for $r$ larger than some (small)
constant $\eps$. On the set $\{x\,:\, d(x,\partial D) > \eps\}
\times \RR^d$, the diffusion \eref{e:diffV} has then a
deterministic component $dK = 0$. In particular, this shows that
\eref{e:diffV} is not hypoelliptic and that its transition
probabilities are not absolutely continuous with respect to
Lebesgue measure.

For any multi-index $\alpha = \{\alpha_1, \ldots, \alpha_\ell\}$,
we define the vector field $\nabla_\alpha V$ by
\begin{equ}
\bigl(\nabla_\alpha V(x)\bigr)_k = {\partial^{\ell +1} V(x) \over
\partial x_k \partial x_{\alpha_1}\ldots
\partial
 x_{\alpha_\ell}}\;.
\end{equ}
Denoting by $|\alpha|$ the size of the multi-index, we
furthermore assume that

\begin{assumption}\label{ass:hypoelliptic}
There exists $x_* \in D$ such that the collection $\{
\nabla_\alpha V(x_*) \}_{|\alpha| > 0}$ spans all of $\RR^d$.
\end{assumption}

\begin{remark}
{\rm In the case where $D$ is smooth and $V$ is of the form $V(x)
= f(d(x, \partial D))$ for some smooth function $f$ diverging at
$0$, Assumption~\ref{ass:hypoelliptic} is satisfied as soon as
there exists a point on the boundary such that its curvature has
full rank.
 }
\end{remark}

We then have

\begin{proposition}\label{prop:unique}
If $e^{-V/2}\in C^\infty(D)\cap W^{1,2}_0(D)$ and Assumption
\ref{ass:hypoelliptic} holds, then $\pi (dx, dy):=c_0 {\bf 1}_D
\exp (-V(x)-|y|^2)\,dx\, dy$ is the unique invariant measure for
\eref{e:diffV}.
\end{proposition}

\begin{proof}
Let us first introduce the following concept. Given a Markov
operator $\CP$ over a Polish space $\CX$, we say that $\CP$ is
\textit{strong Feller at $x$} if $\CP\phi$ is continuous at $x$
for every bounded measurable function $\phi \colon \CX \to \RR$.
The proof of Proposition~\ref{prop:unique} is then based on the
following fact, a version of which can be found for example in
\cite[Thm~3.16]{HM}, which is a consequence of well-known results
from \cite[Chapter 4]{DPZ} and \cite[Section 6.1.2]{MT}: If
$\{\CP_t\}_{t \ge 0}$ is a Markov semigroup over $\CX$ such that
\begin{enumerate}
\item[(i)] There exists $t > 0$ and $x \in \CX$ such that
    $\CP_t$ is strong Feller at $x$,
\item[(ii)] One has $x \in \mathop{supp} \pi$ for every
    invariant probability measure $\pi$ for the semigroup
$\{\CP_t\}$,
\end{enumerate}
then the semigroup $\{\CP_t\}$ can have at most one invariant
probability measure.

Denote now by $\CP_t$ the Markov semigroup generated by solutions
to \eref{e:diffV}. It is easy to check that
Assumption~\ref{ass:hypoelliptic} means precisely that the
operator $\partial_t - \GG $, where $\GG$ is the generator of
\eref{e:diffV}, satisfies H\"ormander's condition
\cite{Hor,Nualart} in some neighborhood of $(x_*, K, t)$ for every
$K \in \RR^d$ and every $t > 0$. Since for any $\phi \in
\BB_b(\RR^d)$, the map $\Phi(t,y,K) = \bigl(\CP_t \phi\bigr)(y,K)$
is a solution (in the sense of distributions) of the equation
$\bigl(\partial_t - \GG\bigr)\Phi = 0$, this implies that $\CP_t
\phi$ is $\CC^\infty$ in a neighborhood of $(x_*, 0)$ for every $t
> 0$. In particular, $\CP_t$ is strong Feller at $(x_*, 0)$ for
every $t > 0$.

It now remains to show that the point $(x_*,0)$ belongs to the
support of every invariant measure of \eref{e:diffV}. This will be
checked by showing first that (ii) follows from the following two
properties:
\begin{enumerate}
 \item[(iii)] for all $y \in \CX$, there exists $s(y) \ge 0$ such
that
 $x_*
 \in \mathop{supp} \CP_{s(y)}(y, \cdot\,)$.
 \item[(iv)]
for every neighborhood $U$ of
$x_*$
 there exists a neighborhood $U' \subset U$ and a time $T_U > 0$
such that $\inf_{y \in U'} \CP_t(y, U) > 0$ for every $t < T_U$.
\end{enumerate}
The argument as to why (iii) and (iv) imply (ii) is the following:

Fix an arbitrary neighborhood $U$ of
$x_*$
 and an arbitrary invariant measure $\pi$. By the definition of an
invariant measure, one then has
\begin{equs}
\pi(U) &= \int_0^\infty \int_{\CX} e^{-t} \CP_t(y,U)\,
\pi(dy)\, dt
\ge \int_{\CX} \int_0^{T_U} e^{-(s(y) + t)} \CP_{s(y) + t}(y,U)\, dt\, \pi(dy) \\
&= \int_{\CX} \int_0^{T_U} \int_{U'} e^{-(s(y) + t)}
\CP_t(y',U) \CP_{s(y)}(y, dy')\, dt\, \pi(dy)
> 0\;.
\end{equs}
In the last inequality, we used the facts that $\CP_{s(y)}(y, U')
> 0$ by (iii) and $\CP_t(y',U) > 0$ by (iv). Since $U$ was
arbitrary, this inequality is precisely (ii).

Since (iv) is satisfied for every SDE with
locally Lipschitz coefficients, it remains to check that (iii) is
also satisfied. For this, it is enough to check that, for every
$(x,K) \in \RR^{2d}$, there exists $T > 0$ such that $\CP_T(x,K;
A) > 0$ for every neighborhood $A$ of the point $(x_*,0)$. In
order to show this, we are going to apply the Stroock-Varadhan
support theorem \cite{SV}, so we consider the control system
\begin{equ}[e:control]
\dot x = -\frac12 \nabla V(x) + K + u(t)\;,\qquad  \dot K =
-\frac12 \nabla V(x)\;,
\end{equ}
where $u \colon [0,T] \to \RR^d$ is a smooth control. The claim is
proved if we can show that for every $(x_0, K_0)$, there exists
$T>0$ such that, for every $\eps > 0$ there exists a control such
that the solution to \eref{e:control} at time $T$ is located in an
$\eps$-neighborhood of the point $(x_*,0)$. Our proof is based on
the fact that, since we assumed that $V(x)$ grows to $+\infty$ as
$x$ approaches $\partial D$, there exists a collection of $\ell$
points $x_1,\ldots, x_\ell$ ($\ell > n$) such that the positive
cone generated by $\nabla V(x_1), \ldots, \nabla V(x_\ell)$ is all
of $\RR^d$. (This fact has already been noted in the paragraph following
\eqref{eqn:2}.) Fix now an initial condition $(x_0, K_0)$. From
the previous argument, there exist positive constants $\alpha_1,
\ldots, \alpha_\ell$ such that $\sum_{i=1}^\ell \alpha_i \nabla
V(x_i) = - K_0$. Fix now $T = \sum_{i=1}^\ell \alpha_i$ and
consider a family $X_\eps \colon [0,T] \to \RR^d$ of smooth
trajectories such that:
\begin{enumerate}
\item[(i)]  there are time intervals $I_i$ of lengths greater than
or equal to $\alpha_i - \eps$ such that $X_\eps(t) = x_i$ for $t
\in I_i$, \item[(ii)] $X_\eps(0) = x_0$ and $X_\eps(T) = x_*$,
\item[(iii)] one has $\int_{[0,T] \setminus \bigcup I_i}
    |\nabla V(X_\eps(t))|\, dt \le \eps$.
\end{enumerate}
Such a family of trajectories can easily be constructed (take a
single trajectory going through all the $x_i$ and run through it
at the appropriate speed). It now suffices to take as control
\begin{equ}
u(t) = \dot X_\eps(t) + \frac12 \nabla V(X_\eps(t)) + \frac12
\int_0^t \nabla V(X_\eps(s))\;,
\end{equ}
and to note that the solution to \eref{e:control} is given by
$\bigl(X_\eps(t), K_0 - \int_0^t \nabla V(X_\eps(s))\bigr)$. This
solution has the desired properties by construction.
\end{proof}

\begin{remark}
{\rm The paper \cite{BRO} also shows uniqueness of the invariant
measure for a mechanical system for which hypoellipticity fails to
hold on an open set. However, the techniques used there seem to be
less straightforward than the argument used here and cannot easily
be ported to our setting.}
\end{remark}

\begin{remark}
{\rm
 Through all of this section, the assumption that $\exp(-V/2) \in
W_0^{1,2}(D)$ could have been replaced by the assumptions that $V
\in \CC^\infty$, $\lim_{x \to
\partial D} V(x) = +\infty$, $\int_{D} \exp(-V(x))\, dx
< \infty$, and
there exists a constant $C \in \RR$ such that the relation
\begin{equ}[e:assV]
\Delta V(x) \le |\nabla V(x)|^2 + C \bigl(1+
V(x)\bigr)
\end{equ}
holds for every $x \in D$. This is because \eref{e:assV} ensures
that $V(x) + {K^2\over 2}$ is a Lyapunov function for the
diffusion \eref{e:diffV} and so guarantees the non-explosion of
solutions.
 }
\end{remark}

\section{Weak convergence}
\label{sec:weakconv}

In this section, we will prove that reflecting diffusions with
inert drift can be approximated by diffusions with smooth inert
drifts.

Let $D$ be a bounded Lipschitz domain in $\RR^d$ and use
$\delta_D(x)$ to denote the Euclidean distance between $x$ and
$D^c$. A continuous function $x\mapsto \delta (x)$ is called a
regularized distance function to $D^c$ if there are constants
$c_2>c_1>0$ such that
\begin{description}
\item{\rm (i)} $ c_1\delta_D(x)\leq \delta (x)\leq
 c_2
  \delta_D(x)$ for every $x\in \R^d$;
\item{\rm (ii)} $\delta \in C^\infty (D)$ and for any
    multi-index $\alpha=(\alpha_1, \cdots, \alpha_d)$ with
    $|\alpha|:=\sum_{k=1}^d \alpha_k$, there is a constant
$c_\alpha$
 such that
$$ \left| \frac{\partial^{|\alpha|} \delta (x)}{(\partial x_1)^{\alpha_1}
 \cdots (\partial x_d)^{\alpha_d}} \right| \leq c_\alpha \delta_D(x)^{1-|\alpha|}
 \qquad \hbox{for every } x\in D.
$$
\end{description}
The existence of such a regularized distance function $\delta$ is
given by
  \cite[Lemma 2.1]{WZ}.
For $n\geq 1$, define
\begin{equ}[eq:50]
V_n(x):=
\begin{cases} \exp \left( \frac1{n \delta (x) } \right)
\qquad &\hbox{for } x\in D, \\
+\infty &\hbox{for } x\in D^c.
\end{cases}
\end{equ}
Observe that $V_n\in C^\infty (D)$ and for every multi-index
$\alpha$, the $\alpha$-derivative of $e^{-V_n}$ at $x\in D$ tends
to zero as $x\to \partial D$. So $e^{-V_n} \in C^\infty (\R^d)$.
Observe also that as $n\to \infty$, $e^{-V_n(x)}$ increases to
$e^{-1}\, \1_D (x)$.

Suppose that $A=\sigma^T\sigma$ is a uniformly elliptic and
bounded matrix
 having $C^2$ entries $a_{ij}$ on $\overline D$, and $\rho$
is a $C^2$ function on $\overline D$ that is bounded
between two positive constants.

 We will consider the smooth potential approximations to the
 symmetric reflecting diffusion process on $D$ associated
with generator ${\cal L}=\frac 1{2\rho}  \nabla (\rho A\nabla)$
but with an additional
inert
 drift. More precisely, let $\Gamma$ be a positive definite
constant symmetric $d\times d$-matrix. As before, we use ${\bf
n}(x)$ to denote the unit inward normal vector at $x\in \partial
D$ and define ${\bf u}(x):= A(x) \n (x)$. We consider the
following special case of the
SDE \eqref{eqn:1}:
\begin{equs}\label{e:diffV1R}
\begin{cases}
dX_t &= \sigma (X_t) \,dB_t +{\bf b}(X_t) \,dt +{\bf u} (X_s) \,
 \,dL_t
 + K_t\, \,dt \;, \\
 dK_t &=  \Gamma \, \n (X_t)\,  d L_t ,
 \end{cases}
\end{equs}
where  $X_t\in \overline D$ for every $t\geq 0$, and $L$ is
continuous and non-decreasing with $L_t=\int_0^1 \1_{\partial
D}(X_s) \,dL_s$. Here $B$ is a $d$-dimensional Brownian motion and
the drift ${\bf b}$ is given by ${\bf b}=(b_1, \cdots, b_d)$ with
 $b_k(x):=\frac1{2\rho} \sum_{i=1}^d
\partial_i ( \rho (x) a_{ik}(x))$.

Let $X^{(n)}$ be the diffusion given by
$$ dX^{(n)}_t= \sigma (X^{(n)}_t)
\,dB^{(n)}_t
 + {\bf b } (X^{(n)}_t) \,dt
 -\frac12 (A\nabla V_n) (X^{(n)}_t)  \,dt  ,
$$
with initial probability distribution $\pi_n(dx):= c_n \exp
(-V_n(x)) \rho (x) \1_D(x) \,dx$, where $B^{(n)}$ is a
$d$-dimensional Brownian motion, and $c_n>0$ is a normalizing
constant so that $\pi_n(\RR^d)=1$. Hence $X^{(n)}$ is the
minimal diffusion in $D$ with infinitesimal generator
$$
\LL^{(n)}=\frac1{2\rho (x)e^{-V_n(x)}} \sum_{i,j=1}^d
\partial_i \left(\rho (x) e^{-V_n (x)} a_{ij} (x) \partial_j \right) ,
$$
which is symmetric with respect to the measure $\pi_n$. Clearly
$e^{(\log \rho -V_n)/2}\in C(D)\cap W^{1,2}_0(D)$ and so by Theorem
\ref{T:4.2}, $X^{(n)}$ is conservative and never reaches $\partial
D$. In addition, $X^{(n)}$ is a symmetric diffusion with respect to
the measure $\pi_n$ in $D$ and its Dirichlet form $(\EE^{(n)},
\FF^{(n)})$ in $L^2(D, \pi_n(dx))$ is given by (see \cite[Lemma
3.5]{WZ})
\begin{equs}
\EE^{(n)}(f, f)&:=\frac12 \int_{D} \sum_{i, j=1}^d a_{ij}(x)
\partial_i f(x) \partial_j f(x) \pi_n(dx)\;,\\
\FF^{(n)}&:=\left\{f\in L^2(D, \pi_n): \EE^{(n)} (f,
f)<\infty\right\}.
\end{equs}
Without loss of generality, we assume that $X^{(n)}$ is defined on
the canonical path space $\Omega:=C([0, \infty), \RR^d)$ with
$X^{(n)}_t (\omega)=\omega (t)$. On $\Omega$, for every $t>0$, there
is a time-reversal operator $r_t$, defined for $\omega\in\Omega$ by
\begin{equation}\label{eqn:timerevers}
r_t(\omega)(s):=
\begin{cases}
\omega (t-s), & \hbox{if }0\le s\le t,\\
\omega(0), & \hbox{if } s\ge t.
\end{cases}
\end{equation}
Let $\P_n$ denote the law of $X^{(n)}$ with initial distribution
$\pi_n$ and let $\{\FF^{(n)}_t, t\geq 0\}$ denote the minimal
augmented filtration generated by $X^{(n)}$. Then it follows from
the reversibility of $X^{(n)}$ that
 for every $t>0$, the measure $\P_n$ on $\FF^{(n)}_t$ is invariant
under the time-reversal operator $r_t$ (cf. \cite[Lemma
5.7.1]{FOT}). We know that
\begin{equation}\label{eqn:a1}
 X^{(n)}_t-X^{(n)}_0=M^{(n)}_t +\int_0^t \left( {\bf b} (X^{(n)}_s)
 -\tfrac12
(A \nabla V_n)
 (X^{(n)}_s ) \right) \,ds, \qquad t\geq 0,
\end{equation}
where $M^{(n)}_t:=\int_0^t \sigma (X^{(n)}_s ) \,dB^{(n)}_s$ is a
square-integrable martingale.
 On the other
hand, by the Lyons-Zheng's forward and backward martingale
decomposition (see \cite[Theorem 5.7.1]{FOT}), we have for every
$T>0$,
\begin{equation}\label{eqn:a2}
  X^{(n)}_t-X^{(n)}_0 = \frac12 M^{(n)}_t - \frac12 (M^{(n)}_T-M_{T-t}^{(n)} ) \circ r_T
\qquad \hbox{for } t\in [0, T].
\end{equation}
Since $\sigma$ is bounded, for every $T>0$, the law
of $\{M^{(n)}_t, t\geq 0\}$ under $\P_n$ is
 tight in the space $C([0, T], \RR^d)$ and so is $\frac12
(M^{(n)}_T-M_{T-t}^{(n)} ) \circ r_T$. It follows that the laws
 of  $(X^{(n)}, M^{(n)}, (M^{(n)}_T-M_{T-\cdot}^{(n)} ) \circ r_T)$ under $\P_n$
are tight in the space $C([0, T], \RR^{3d})$ for every $T>0$. On
the other hand, by (\ref{eqn:a1})-(\ref{eqn:a2}),
$$ H^{(n)}(t):=-\frac12   \int_0^t  (A \nabla V_n) (X^{(n)}_s)\,ds=-\frac12
 B^{(n)}_t- \frac12  (B^{(n)}_T-B_{T-t}^{(n)} ) \circ r_T -\int_0^t
{\bf b} (X^{(n)}_s) \,ds.
$$
Thus the laws of $(X^{(n)}, M^{(n)}, \, \int_0^\cdot {\bf b}
(X^{(n)}_s ) \,ds, \, H^{(n)}, \, B^{(n)})$ under $\P_n$ are tight
with respect to  the space $C([0, T], \RR^{5d})$ for every $T>0$. By almost
the same argument as that for \cite[Theorems 4.1 and 4.4]{PW},
$(X^{(n)}, \P_n)$ converges weakly in $C([0, T], \RR^d)$ to
a stationary reflecting diffusion  $X$ in $D$ with normalized
initial distribution $c \rho (x) \,dx$ on $D$. Passing to a
subsequence, if necessary, we conclude that
\begin{description}
\item{}
\qquad $((X^{(n)},M^{(n)}, \,
\int_0^\cdot {\bf b} (X^{(n)}_s ) \,ds, \, H^{(n)}, \, B^{(n)}),
\P_n)$ converges weakly in $C([0, T], \RR^{4d})$ to a process
 \hfil\break
$((X,M,  \, \int_0^\cdot {\bf b} (X_s ) \,ds, \, H, \,
B), \P)$.
\end{description}

Now we apply the Skorokhod representation (see Theorem 3.1.8 in \cite{EK})
to construct the processes $(X^{(n)},M^{(n)}, \, \int_0^\cdot {\bf
b} (X^{(n)}_s ) \,ds, \, H^{(n)}, \, B^{(n)})$ and $(X,M,  \,
\int_0^\cdot {\bf b} (X_s ) \,ds, \, H, \, B)$ on the same
probability space $(\Omega, {\cal F}, \P)$ in such a way that
$(X^{(n)},M^{(n)}, \, \int_0^\cdot {\bf b} (X^{(n)}_s ) \,ds, \,
H^{(n)}, \, B^{(n)})$ converges to \hfill \break $(X,M,  \,
\int_0^\cdot {\bf b} (X_s ) \,ds, \, H, \, B)$ a.s., on the time
interval $[0, 1]$ in the supremum norm. Therefore, in view of
(\ref{eqn:a1}),
 $$ X_t=X_0+M_t +  \int_0^t {\bf b} (X_s ) \,ds +H
 \qquad \hbox{for every } t\geq 0.
$$

By the proof of \cite[Theorem 6.1 and remark 6.2]{PW}, $H$ is a
continuous process of locally finite variation. On the other hand,
since $X$ is reflecting diffusion in a Lipschitz domain, by
\cite{C} it admits a Skorokhod decomposition. That is,
$$ X_t=X_0+\int_0^t \sigma (X_s) \,dB_t+ \int_0^t {\bf b}(X_s) \,ds+
\int_0^t A(X_s) \n (X_s)\,dL_s, \qquad t\geq 0,
$$
where $B$ is a Brownian motion with respect to the minimal
augmented filtration generated by $X_t$, $\n$ is the unit inward
normal vector field on $\partial D$, which is well defined a.e.
with respect to the surface measure $\sigma$, and $L$ is a
positive continuous additive functional of $X$  with Revuz measure
$\sigma$. By the uniqueness of Doob-Meyer decomposition, we have
$$ H_t=\int_0^t A (X_s) \n (X_s) \,dL_s, \qquad \hbox{for } t\geq 0.
$$

Define for $t\geq 0$,
$$K^{(n)}_t:=-\frac12  \Gamma \, \int_0^t  \nabla V_n (X^{(n)}_s) \,ds
\qquad \hbox{and} \qquad K_t:= \Gamma \, \int_0^t \n (X_s) \,dL_s.
$$
 Set
$$ Z_t=\exp\left( \int_0^t  \sigma^{-1} (X_s) K_s \,dB_s-\frac 12 \int_0^t
| \sigma^{-1} (X_s)K_s|^2 \,ds \right), \quad t\geq 0,
$$
and
$$ Z^{(n)}_t=\exp\left( \int_0^t \sigma^{-1}(X^{(n)}_s) K^{(n)}_s \,dB^{(n)}_s
-\frac 12 \int_0^t |\sigma^{-1}(X^{(n)}_s) K^{(n)}_s|^2 \,ds
\right), \qquad t\geq 0.
$$
Let $d\Q = Z_1 d\P$ and $d\Q_n = Z^{(n)}_1 d\P$. By the Girsanov
theorem, under $\Q_n$, $(X^{(n)}, K^{(n)})$ satisfies the
following equation
\begin{equs}
\begin{cases}
dX^{(n)}_t= \sigma (X^{(n)}_t) \,dW^{(n)}_t+{\bf b} (X^{(n)}_t) -\frac12 (A \nabla V_n) (X^{(n)}_t) \,dt +K^{(n)}_t \,dt, \\
dK^{(n)}_t= -\frac12 \Gamma \, \nabla V_n (X^{(n)}_t) \,dt ,
\end{cases}
\end{equs}
where $W^{(n)}$ is a $d$-dimensional Brownian motion. On the other
hand, by the proof of Theorem  \ref{T:2.1}, $(X, K)$ under $\Q$ is
Brownian motion with inert drift, and satisfies
\begin{equs}
\begin{cases}
dX_t= \sigma (X_t) \,dW_t+ {\bf b}(X_t) \,dt + (A\n) (X_t) \,dL_t +K_t \,dt, \\
dK_t= \Gamma \, \n(X_t) \,dL_t ,
\end{cases}
\end{equs}
where $W$ is a $d$-dimensional Brownian motion.

In the following, the initial distribution of $(X^{(n)}_0,
K^{(n)}_0)$ and $(X_0, K_0)$ under $\P$ are taken to be
\begin{equ}[defmu]
c_n \1_D(x) \rho (x) e^{-V_n(x)-(\Gamma^{-1}y, y)} \,dx \,dy
\end{equ}
and  $c \1_D(x) \rho (x) e^{-(\Gamma^{-1}y, y)}\, dx\, dy$, respectively, where
$c_n>0$ and $c>0$ are normalizing constants chosen to make sure that
the corresponding measures are probability measures. Since $e^{(\log
\rho -V_n)/2}\in C(D)\cap W^{1,2}_0(D)$, we know from Theorem
\ref{T:4.3} that $c_n \1_D(x) \rho (x) e^{-V_n(x)-(\Gamma^{-1}y, y)}
\,dx \,dy$ is a stationary measure for $(X^{(n)}, K^{(n)})$ under $\Q_n$.

\begin{theorem}\label{T:wc.1}
For every $T>0$, the law of $(X^{(n)}, K^{(n)})$ under $\Q_n$
converges weakly in the space $C([0, T], \RR^{2d})$ to that of the
law of $(X, K)$ under $\Q$.
\end{theorem}

\begin{proof}
Without loss of generality, we take $T=1$.

Observe that $A(X^{(n)})$ converges to $A(X)$  uniformly on $[0,1]$
and $K^{(n)}_t = \Gamma \, \int_0^t A^{-1}(X^{(n)}_s) dH^{(n)}_s$.
Let $|H^{(n)}|_t$ and $|K^{(n)}|_t$  denote the total variation
process of $H^{(n)}$ and $K^{(n)}$   over the interval $[0, t]$,
respectively. Then there is a constant $c>0$ independent of $n\geq 1$
such that
$$ |K^{(n)}|_t \leq |H^{(n)}|_t \qquad \hbox{for every } t\geq 0.
$$
On the other hand, we know from the proof of Theorem 6.1 in
\cite[p.58-59]{PW} that
$$ \sup_{n\geq 1}
\E_{\P_n}
 \left[ |H^{(n)}|_t\right]<\infty \qquad \hbox{for every } t\geq 0.
$$
Hence by Theorem 2.2 of \cite{KP}, $K^{(n)}$ converges to $K$ in
probability with respect to the uniform topology on $C[0, 1]$.
This
yields, by \cite[Theorem 2.2]{KP} again, that
 $\int_0^1 \sigma^{-1}(X^{(n)}_s) K^{(n)}_s \,dB^{(n)}_s \to \int_0^1
\sigma^{-1}(X_s) K_s \,dB_s$ as $n\to \infty$, in probability. By
passing to a subsequence, we may assume that the convergence is
$\P$-almost sure. We conclude that $\lim_{n\to \infty} Z^{(n)}_1=
Z_1 $, $\P$-a.s.

Let $\Phi$ be a continuous function on $(C[0, 1])^2$ with $0\leq
\Phi\leq 1$. Since $\Phi (X^{(n)}, K^{(n)})\to \Phi(X, K)$,
$\P$-a.s., and $Z^{(n)}_1\to Z_1$, $\P$-a.s.,
 by Fatou's lemma,
\begin{equ}[eqn:51]
\E_\P
 [ \Phi (X, K) Z_1]\leq \liminf_{n\to \infty}
\E_\P
 [ \Phi (X^{(n)}, K^{(n)}) Z^{(n)}_1] \leq \limsup_{n\to \infty}
\E_\P
 [ \Phi (X^{(n)}, K^{(n)}) Z^{(n)}_1]
\end{equ}
and
\begin{equ}[eqn:52]
\E_\P
 [(1- \Phi) (X, K) Z_1]\leq \liminf_{n\to \infty}
\E_\P
 [ (1-\Phi) (X^{(n)}, K^{(n)}) Z^{(n)}_1] .
\end{equ}
Summing \eqref{eqn:51} and \eqref{eqn:52} we obtain $\E_\P[Z_1] \leq
\limsup_{n\to \infty} \E_\P [Z^{(n)}_1]$. Note that by the proof of
Theorem \ref{T:4.3}, $Z^{(n)}$ is a continuous non-negative
$\P$-martingale, while by the proof of Theorem \ref{T:2.1}, $M$ is a
continuous $\P$-martingale. Hence, $\E_\P[Z_1]=1= \E_\P [Z^{(n)}_1]$
and, therefore the inequalities in \eqref{eqn:51} and \eqref{eqn:52}
are in fact equalities. It follows that
 $$ \lim_{n\to \infty} \bQ_n [ \Phi (X^{(n)}, K^{(n)}) ]=
 \lim_{n\to \infty}
 \E_\P
 [ \Phi (X^{(n)}, K^{(n)}) Z^{(n)}_1] =
 \E_\P
 [ \Phi (X, K) Z_1] =\bQ [ \Phi (X, K) ].
 $$
This proves the weak convergence of
 $(X^{(n)}, K^{(n)})$ under $\Q_n$ to $(X, K)$ under $\Q$.
\end{proof}

\begin{theorem}\label{T:wc.3}
Consider the SDE \eqref{e:diffV1R} on a bounded Lipschitz
domain $D\subset \R^d$. A stationary distribution for the
solution $(X, K)$ to \eqref{e:diffV1R} is
$$\pi (dx, dy)= c_1 \1_D(x) \rho (x) e^{-(\Gamma^{-1}y, y)}\, dx\, dy,
$$
where $c_1$ is the normalizing constant so that $\pi (\overline
D\times \R^d)=1$.
\end{theorem}

\begin{proof}
Recall the notation from previous proofs in this section. We start
$(X^{(n)},K^{(n)})$ with the stationary distribution $\pi_n$ of
Theorem \ref{T:4.3}, namely,
$$\Q_n((X^{(n)}_0,K^{(n)}_0)\in A)=\pi_n(A)=
c_n
 \int_A e^{-V_n(x)} \rho (x) e^{-(\Gamma^{-1}y, y)}\, dx\, dy, \qquad A\subset
 \overline{D}\times \R^d,
 $$
where $c_n$ is a normalizing constant. Note that here we have
$e^{-V_n(x)} \rho (x)$ in place of $e^{-V(x)}$ in Theorem
\ref{T:1}. Clearly $\pi_n$ converge weakly on $\overline D\times
\RR^d$ to the probability $\pi$. Also, the $\Q_n$ laws of
$(X^{(n)},B^{(n)},K^{(n)})$ converge weakly to the $\Q$ law of
$(X,B,K)$. If $f$ is any continuous and bounded function on
$\R^{2d}$, then for $t_0\leq 1$,
\begin{align*}
\E_\Q
 f(X_{t_0},K_{t_0})&=\lim_{n\to\infty}
 \E_{\Q_n} f(X^{(n)}_{t_0},K^{(n)}_{t_0})
=\lim_{n\to \infty} \int_{\overline D \times \RR^d} f(x, y) \pi_n
(dx, dy)\\
&=\int_{\overline D \times \RR^d} f(x, y) \pi (dx, dy).
\end{align*}
This shows that the $\pi$  is a stationary distribution for the
solution to \eqref{eqn:1}.
\end{proof}

\section{Irreducibility}
\label{sec:irreducibility}

In this section, $D$ is a bounded $C^2$ domain in $\RR^d$. We will
prove uniqueness of the stationary distribution for $(X,K)$ under
assumptions stronger than those in previous sections. Let
$\Q_{x,y}$ denote the distribution of $(X,K)$ with inert drift
starting from $(x,y)$. We say that $(X,K)$ is irreducible in the
sense of Harris if there exists a positive measure $\mu$ on $\ol D
\times \RR^d$ and $t_0>0$ such that if $\mu (A) > 0$, then for all
$(x,y) \in \ol D\times \RR^d$, $\Q_{x,y} ((X_{t_0},K_{t_0}) \in A)
> 0$.

Let $\Gamma$ be a positive definite constant symmetric $d\times
d$-matrix.

\begin{thm}\label{T:irred}
Assume that $\sigma$ is the identity matrix (consequently $\cn
\equiv \n$), ${\bf b} = {1\over 2} \nabla \log \rho$, and $\rho$
is a $C^2$ function on $\overline D$ that is bounded between two
positive constants. Then the solution $(X,K)$ to \eqref{e:diffV1R}
is irreducible in the sense of Harris.
\end{thm}

\begin{proof}
We first assume $\rho \equiv 1$ on $\overline D$. Let $X$ denote a
solution to \eqref{eqn:2.2}, i.e., the usual normally reflecting
Brownian motion in $D$, let $L$ be the local time of $X$ on $\prt
D$, and $K_t=y_0+\int_0^s \n (X_s) \,dL_s$. The distribution of
$(X,K)$ starting from $(X_0,K_0) = (x,y)$ will be denoted
$\P_{x,y}$. First, we will prove Harris irreducibility for $(X,K)$
under $\P_{x,y}$, with respect to $2d$-dimensional Lebesgue
measure.

In order to do so,  our main ingredient will be that for $t>0$ the
law of $(X_t, K_t)$ has a component that has a strictly positive
density with respect to Lebesgue measure on some open set. The idea
behind the proof of this result is that one can find $d$ small balls
$\{B_j\}_{j=1}^d$ on the boundary $\partial D$ such that
\begin{itemize}
 \item[(i)]
The normal vector $\n(x)$ is `almost' parallel to the $j$-th unit vector $e_j$
for $x \in B_j$.
\item[(ii)] With positive probability, the process $X_s$ has
    visited all of the $B_j$'s in chronological order before time
    $t$, but has not hit the boundary $\prt D$ anywhere outside
    of the $B_j$'s.
\item[(iii)] The amounts $S_j$ of local time that $X$ has accumulated
on the $B_j$'s are `almost' independent.
\end{itemize}
This suggests that at time $t$, the law of $K_t$ has a positive
component which `almost' looks like the law of a random vector with
independent components, each of them having a density with respect to
Lebesgue measure. It follows that the law of $K_t$ has a component
which has a density with respect to $d$-dimensional Lebesgue measure.
Since, as long as $X$ is in the interior of the domain $D$, it is
just Brownian motion with drift and $K$ remains constant, we conclude
that the law of $(X_t, K_t)$ has a density with respect to
$2d$-dimensional Lebesgue measure.

The detailed proof is broken into several distinct steps. In the
first step, we use a support theorem and a controllability
argument to show that given any point $z_1 \in D$, $\eps >0$ and
any final time $t_0$, the process $(X_t, K_t)$ has a positive
probability to be in an $\eps$-neighborhood of $(z_1, 0)$ after
time $t_0/2$, whatever its initial condition. In the second step,
we present a review of excursion theory and show how the path of a
reflecting Brownian motion can be decomposed into a collection of
excursions and how reflecting Brownian motion up to the first
hitting time of some subset $U \subset D$ can be constructed from
the excursions of a reflecting Brownian motion conditioned never
to hit $U$ by adding a `last excursion' after a suitably chosen
amount of local time spent at the boundary. This construction is
then used in the third step to `stitch together' a reflecting
Brownian motion from $d$ independent reflecting Brownian motions
$Y_t^j$. In the final step, we show how to condition each of the
$Y^j$'s on hitting the boundary $\prt D$ only in $B_j$ and deduce
from this that the local time $K_t$ has a density with respect to
Lebesgue measure. We conclude by showing how to combine these
results to obtain the desired Harris irreducibility.

\medskip
\noindent{\it Step 1}. Fix any $t_0,r >0$ and $z_1\in D$. In
this step, we will show that for any $(x_0,y_0) \in \ol D\times
\RR^d$ there exists $p_1>0$ such that $\P_{x_0,y_0}
((X_{t_0/2},K_{t_0/2}) \in \BB(z_1,r)\times \BB(0,r)) \geq
p_1$.

We recall the deterministic Skorokhod problem in
$D$ with normal vector of reflection. Suppose a continuous
function $f: [0,T] \to \RR^d$ is such that $f(0) \in \ol D$.
Then the Skorokhod problem is to find a continuous function $g:
[0,T] \to \ol D$ and a non-decreasing function $\ell: [0,T] \to
[0,\infty)$, such that $\ell(0) = 0$, $g(0) = f(0)$, $\int_0^T
\1_D (g(s)) \,d\ell_s =0$, and $g(t) = f(t) + \int_0^t \n(g(s))
\,d\ell_s$. It has been proved in \cite{LS} that the Skorokhod
problem has a unique solution $(g,\ell)$ in every $C^2$ domain.

Since $D$ is a bounded smooth domain, the set $\{\n(x)/|\n(x)|,
x\in \prt D\}$ is the whole unit sphere in $\RR^d$. Find $x_1 \in
\prt D$ such that $\n(x_1) = - c_0 y_0$ for some $c_0 >0$. It is
elementary to construct a continuous function $f:[0,t_0/2] \to
\RR^d$ such that $f(0) = x_0$, $f(t) \in D$ for $t\in (0,t_0/4)$,
$f(t_0/4) = x_1$, $f$ is linear on $[t_0/4, 3 t_0/8]$, $f(3t_0/8)
= x_1 +y_0$, $f(t) - y_0\in D$ for $t\in (3t_0/8,t_0/2)$, and
$f(t_0/2) -y_0 = z_1$. It is straightforward to check that the
pair $(g,\ell)$ that solves the Skorokhod problem for $f$ has the
following properties: $g(t) = f(t)$ for $t\in (0,t_0/4)$, $g(t) =
x_1$ for $t\in [t_0/4, 3 t_0/8]$, $g(t) = f(t) - y_0$ for $t\in
(3t_0/8,t_0/2)$, and $\int_0^{t_0/2} \n(g(s)) \,d\ell_s = -y_0$.

For a function $f^1:[0,t_0/2] \to \RR^d$ with $f^1(0) \in \ol
D$, let $(g^1, \ell^1)$ denote the solution of the Skorokhod
problem for $f^1$. Let $\BB_C(f,\delta)$ be the ball in
$C([0,t_0/2], \RR^d)$ centered at $f$, with radius $\delta$, in
the supremum norm. By Theorem 2.1 and Remark 2.1 of \cite{LS},
for any $r>0$ there exists $\delta\in (0,r/2)$, such that if
$f^1 \in \BB_C(f,\delta)$, then $g^1 \in \BB_C(g,r/2)$. This
implies that
\begin{align*}
 \sup_{t\in[0,t_0/2]} \Big|\int_0^{t_0/2} \n(g^1(s)) \,d\ell^1_s
 &- \int_0^{t_0/2} \n(g(s)) \,d\ell_s \Big|\\
 &\leq \sup_{t\in[0,t_0/2]} (|f^1(t) - f(t)| + |g^1(t) - g(t)|)\\
 &\leq \delta + r/2 \leq r.
\end{align*}
Thus, if $f^1 \in \BB_C(f,\delta)$, then
$\left(g^1(t_0/2),\int_0^{t_0/2} \n(g^1(s)) \,d\ell^1_s\right)
\in \BB(z_1,r)\times \BB(-y_0,r) $. Let $\widetilde \P_x$
denote the distribution of standard Brownian motion. By the
support theorem for Brownian motion, $\widetilde \P_{x_0}
(\BB_C(f,\delta)) > p_1$, for some $p_1>0$. Hence,
$\P_{x_0,y_0} ((X_{t_0/2},K_{t_0/2}) \in \BB(z_1,r)\times
\BB(0,r))> p_1>0$.

\medskip
\noindent{\it Step 2}. This step is mostly a review of the
excursion theory needed in the rest of the argument. See, e.g.,
\cite{M} for the foundations of excursion theory in abstract
settings and \cite{B} for the special case of excursions of
Brownian motion. See also \cite{H} for excursions of reflecting
Brownian motion on $C^3$ domains. Although \cite{B} does
not discuss reflecting Brownian motion, all the results we need
from that book readily apply in the present context. We will use
two different but closely related exit systems. The first one
represents excursions of reflecting Brownian motion from $\prt D$.

We consider $X$ under a probability measure $\P_x$,
i.e., $X$ denotes reflecting Brownian motion without inert
drift.

An exit system for excursions of reflecting Brownian motion $X$
from $\prt D$ is a pair $(L^*_t, {\bf H}_x )$ consisting of a
positive continuous additive functional $L^*_t$ and a family of
excursion laws $\{{\bf H}_x\}_{x\in\prt D}$. We will soon show
that $L^*_t = L_t$. Let $\Delta$ denote a ``cemetery'' point
outside $\RR^d$ and let ${\cal C}$ be the space of all functions
$f:[0,\infty) \to \RR^d\cup\{\Delta\}$ which are continuous and
take values in $\RR^d$ on some interval $[0,\zeta)$, and are equal
to $\Delta$ on $[\zeta,\infty)$. For $x\in \prt D$, the excursion
law ${\bf H}_x$ is a $\sigma$-finite (positive) measure on $\cal
C$, such that the canonical process is strong Markov on
$(t_0,\infty)$ for every $t_0>0$, with the transition
probabilities of Brownian motion killed upon hitting $\prt D$.
Moreover, ${\bf H}_x$ gives zero mass to paths which do not start
from $x$. We will be concerned only with ``standard'' excursion
laws; see Definition 3.2 of \cite{B}. For every $x\in \prt D$
there exists a standard excursion law ${\bf H}_x$ in $D$, unique
up to a multiplicative constant.

Excursions of $X$ from $\prt D$ will be denoted $e$ or $e_s$,
i.e., if $s< u$, $X_s,X_u\in\prt D$, and $X_t \notin \prt D$ for
$t\in(s,u)$, then $e_s = \{e_s(t) = X_{t+s} ,\, t\in[0,u-s)\}$ and
$\zeta(e_s) = u -s$. By convention, $e_s(t) = \Delta$ for $t\geq
\zeta$. So $e_t \equiv \Delta$ if $\inf\{r> t: X_r \in \prt D\} =
t$. Let ${\cal E}_u = \{e_s: s \leq u\}$.

Let $\sigma_t = \inf\{s\geq 0: L^*_s \geq t\}$ and let $I$ be the
set of left endpoints of all connected open components of $(0,
\infty)\setminus \{t\geq 0: X_t\in \partial D\}$. The following is
a special case of the exit system formula of \cite{M}. For every
$x\in \ol D$,
 \begin{equation}\label{es1}
 \E_x \left[ \sum_{t\in I} Z_t \cdot f ( e_t) \right]
 = \E_x \left[ \int_0^\infty Z_{\sigma_s}
 {\bf H}_{X(\sigma_s)}(f) \,ds \right]
 = \E_x \left[ \int_0^\infty Z_t {\bf H}_{X_t}(f) \,dL^*_t \right],
 \end{equation}
where $Z_t$ is a predictable process and $f:\, {\cal
C}\to[0,\infty)$ is a universally measurable function which
vanishes on those excursions $e_t$ identically equal to $\Delta$.
Here and elsewhere ${\bf H}_x(f) = \int_{\cal C} f d{\bf H}_x$.

The normalization of the exit system is somewhat arbitrary, for
example, if $(L^*_t, {\bf H}_x)$ is an exit system and
$c\in(0,\infty)$ is a constant then $(cL^*_t, (1/c){\bf H}_x)$ is
also an exit system. Let $\P_y^D$ denote the distribution of
Brownian motion starting from $y$ and killed upon exiting $D$.
Theorem 7.2 of \cite{B} shows how to choose a ``canonical''
exit system; that theorem is stated for the usual planar
Brownian motion but it is easy to check that both the statement
and the proof apply to reflecting Brownian motion in $D
\subset \RR^d$. According to that result, we can take $L^*_t$
to be the continuous additive functional whose Revuz measure is
a constant multiple of the surface area measure $dx$ on $\prt
D$ and the ${\bf H}_x$'s to be standard excursion laws normalized so that
for some constant $c_1\in(0,\infty)$,
 \begin{equation}\label{es2}
 {\bf H}_x (A) = c_1 \lim_{\delta\downarrow 0} {\frac1\delta} \P^D_{x +
 \delta\n(x)} (A)
 \end{equation}
for any event $A$ in a $\sigma$-field generated by the process
on an interval $[t_0,\infty)$ for any $t_0>0$. The Revuz
measure of $L$ is $c_2 \,dx$ on $\prt D$. We choose $c_1$ so that
$(L_t, {\bf H}_x)$ is an exit system.

We will now discuss another exit system, for a different process
$X'$. Let $U \subset D$ be a fixed closed ball with positive
radius, and let $X'$ be the process $X$ conditioned by the event
$\{T^X_U > \sigma_1\}$, where $T^X_U = \inf\{t>0: X\in U\}$. One
can show using Theorem 2.1 and Remark 2.1 of \cite{LS} that for
any starting point in $D \setminus U$, the probability of $\{T^X_U
> \sigma_1\}$ is greater than 0. It is easy to see that $(X'_t,
L_t)$ is a time-homogeneous Markov process. For notational
consistency, we will write $(X'_t, L'_t)$ instead of $(X'_t,
L_t)$.

We will now describe an exit system $(L'_t, {\bf H}'_{x,\ell})$
for $(X'_t, L'_t)$ from the closed set $\ol D \times [0, \infty)$.
We will construct this exit system on the basis of $(L_t, {\bf
H}_x)$ because of the way that $X'$ has been defined in relation
to $X$. It is clear that $L'$ does not change within any excursion
interval of $X'$ away from $\prt D$, so we will assume that ${\bf
H}'_{x,\ell}$ is a measure on paths representing $X'$ only. The
local time $L'$ is the continuous additive functional with Revuz
measure $c_2 \,dx$ on $\prt D$. For $\ell \geq 1$ we let ${\bf
H}'_{x,\ell} = {\bf H}_x$. Let $\widehat \P^D_y$ denote the
distribution of Brownian motion starting from $y\in D\setminus U$,
conditioned to hit $\prt D$ before hitting $U$, and killed upon
exiting $D$. For $\ell < 1$, we have
 \begin{equation}\label{es3}
 {\bf H}'_{x,\ell} (A) = c_1 \lim_{\delta\downarrow 0} {\frac 1 \delta}
 \widehat \P^D_{x + \delta\n(x)} (A).
 \end{equation}

Let $A_* \subset {\cal C}$ be the event that the path hits $U$.
It follows from (\ref{es2}) and (\ref{es3}) that
for $\ell < 1$,
 \begin{equation}\label{es4}
 {\bf H}'_{x,\ell}(A) = {\bf H}_x(A \setminus A_*).
 \end{equation}
One can deduce easily from (\ref{es2}) and standard estimates
for Brownian motion that for some $c_3,c_4 \in (0,\infty)$ and
all $x\in \prt D$,
 \begin{equation}\label{es5}
 c_3 < {\bf H}_x (A_*) < c_4.
 \end{equation}
Let $\sigma'_t = \inf\{s\geq 0: L'_s \geq t\}$. The exit system
formula (\ref{es1}) and (\ref{es4}) imply that we can construct
$X$ (on a random interval, to be specified below) using $X'$ as a
building block, in the following way. Suppose that $X'$ is given.
We enlarge the probability space, if necessary, and construct a
Poisson point process $\EE$ with state space $[0,\infty) \times
{\cal C}$ whose intensity measure conditional on the whole
trajectory $\{X'_t, t\geq 0\}$ is given by
\begin{equation}\label{e:mu}
 \mu([s_1, s_2] \times F)
 = \int_{1\land s_1}^{1\land s_2}
 {\bf H}_{X'_{\sigma'_t}}(F \cap A_*) \,dt.
 \end{equation}
Since $\mu([0,\infty) \times {\cal C}) < \infty$, the Poisson
point process $\EE$ may be empty; that is, if the Poisson process
is viewed as a random measure, then the support of that measure
may be empty. Consider the case when it is not empty and let $S_1$
be the minimum of the first coordinates of points in $\EE$. Note
that there can be only one point $(S_1,e_{S_1})\in \EE$ with first
coordinate $S_1$, because of (\ref{es5}). By convention, let $S_1
=\infty$ if $\EE = \emptyset$. Recall that $T^X_U = \inf\{t>0: X_t
\in U\}$ and let
\begin{equs}
 T^{X'}_U &= \inf\{t>0: X'_t \in U\}, \\
 T_* &= \sigma'_{S_1} + \inf\{t>0: e_{S_1} (t) \in U\}.
\end{equs}
It follows from the exit system formula (\ref{es1}) that the
distribution of the process
 $$
 \widehat X_t=
 \begin{cases}
 X'_t & \hbox{if  } 0 \leq t \leq
 T^{X'}_U \land \sigma'_{S_1},  \\
 e_{S_1}(t -\sigma'_{S_1}) & \hbox {if  }
 \EE \ne \emptyset \hbox{  and  }
 \sigma'_{S_1} < t \leq T_*,
 \end{cases}
 $$
is the same as the distribution of $\{X_t, 0 \leq t \leq
T^X_U\}$.

\medskip
\noindent{\it Step 3}. We will now construct reflecting Brownian
motion in $D$ from several trajectories, including a
family of independent paths.

Let $U_j = \ol{ \BB(z_j, r)}$ for $j=1, \dots, d+1$, where $z_j
\in D$ and $r>0$ are chosen so that $U_j \cap U_k = \emptyset$ for
$j\ne k$, and $\bigcup _{1\leq j \leq d+1} U_j \subset D$.

Recall from the last step how the process $X$ was constructed from
a process $X'$. Fix some $x_1 \in U_1$ and let $X^1$ be a process
starting from $X^1_0 = x_1$, with the same transition
probabilities as $X'$, relative to $U_2$. We then construct $Y^1$
based on $X^1$, by adding an excursion that hits $U_2$, in the
same way as $X$ was constructed from $X'$. We thus obtain a
process $\{Y^1_t, 0 \leq t \leq T_1\}$, where $T_1 = \inf\{t>0 :
Y^1_t \in U_2\}$, whose distribution is that of reflecting
Brownian motion in $D$ starting with the uniform distribution on
$U_j$, observed until the first hit of $U_2$.

We next construct a family of independent reflecting Brownian
motions $\{Y^j\}_{1\leq j \leq d}$. For a fixed $j=2, \dots , d$,
we let $X^j$ be a process with the same transition probabilities
as $X'$, relative to $U_{j+1}$, and initial distribution uniform
in $U_j$. We then construct $Y^j$ based on $X^j$, by adding an
excursion that hits $U_{j+1}$, in the same way as $X$ was
constructed from $X'$. We thus obtain a process $\{Y^j_t, 0 \leq t
\leq T_j\}$, where $T_j = \inf\{t>0 : Y^1_t \in U_{j+1}\}$, whose
distribution is that of reflecting Brownian motion in $D$,
observed until the first hit of $U_{j+1}$.

Note that for some $c_5>0$ and all $x,y\in U_{j+1}$, $j=1,
\dots, d$,
 $$\P_x(X_1 \in dy \hbox{ and }  X_t \notin \prt D
 \hbox{ for  } t\in[0,1] )
 \geq c_5 \,dy.
 $$
We can assume that all $X^j$'s and $Y^j$'s are defined on the
same probability space. The last formula and standard coupling
techniques show that on an enlarged probability space, there
exist reflecting Brownian motions $Z^j$, $j=1, \dots, d$, with
the following properties. For $1 \leq j \leq d-1$, $Z^j_0 =
Y^j_{T_j}$, and for some $c_6>0$,
 \begin{equation}\label{es6}
 \P \left( Z^j_1 = Y^{j+1}_0 \hbox{ and } Z^j_t \notin \prt D
 \hbox{ for  } t\in[0,1] \, \Big| \, \{Y^k\}_{1 \leq k \leq j},
 \{Z^k\}_{ 1\leq k \leq j-1}\right)
 \geq c_6.
 \end{equation}
The process $Z^j$ does not depend otherwise on $ \{Y^k\}_{1
\leq k \leq d}$ and $ \{Z^k\}_{ k \ne j}$. We define $Z^d$ as a
reflecting Brownian motion in $D$ with $Z^d_0 = Y^d_{T_d}$ but
otherwise independent of $ \{Y^k\}_{1 \leq k \leq d}$ and $
\{Z^k\}_{ 1\leq k \leq d-1}$.

Let
$$F_j = \left\{Z^j_1 = Y^{j+1}_0 \hbox{ and }
  Z^j_t \notin \prt D \hbox{ for
} t\in[0,1] \right\}.
$$
We define a process $X^*$ as follows. We let
$X^*_t = Y^1_t$ for $0\leq t \leq T_1$. If $F_1^c$ holds, then
we let $X^*_t = Z^1_{t - T_1}$ for $t \geq T_1$. If $F_1$ holds,
then we let $X^*_t = Z^1_{t - T_1}$ for $t \in [T_1, T_1 + 1]$
and $X^*_t = Y^2_{t - T_1 -1}$ for $t\in [T_1+1, T_1 + 1 +
T_2]$. We proceed by induction. Suppose that $X^*_t$ has been
defined so far only for
 $$t\in [0, T_1 + 1 + T_2 + 1 + \dots + T_k],
 $$
for some $k< d$. If $F_k^c$ holds, then we let
 $$X^*_t = Z^k_{t - T_1-1 - T_2 - 1- \dots - T_k}
 $$
for $t \geq T_1 + 1 + T_2 + 1 + \dots + T_k$. If $F_k $ holds,
then we let
 $$X^*_t = Z^k_{t - T_1-1 - T_2 - 1-\dots -T_k}
 $$
for $t \in [T_1 + 1 + T_2 + 1 + \dots + T_k, T_1 + 1 + T_2 + 1
+ \dots + T_k + 1]$ and
 $$X^*_t = Y^{k+1}_{t - T_1 -1 - T_2 -1-\dots - T_k -1}
 $$
for $t\in [T_1 + 1 + T_2 + 1 + \dots + T_k +1, T_1 + 1 + T_2 +
1 + \dots + T_k + 1 + T_{k+1}]$. We let
 $$X^*_t = Z^d_{t - T_1-1- T_2 - 1-\dots - T_d}
 $$
for $t \geq T_1 + 1 + T_2 + 1 + \dots + T_d$.

By construction, $X^*$ is a reflecting Brownian motion in $D$
starting from $x_1$. Note that in view of (\ref{es6}),
conditional on $\{ X^j_t, t\geq 0\}$, $ j = 1 , \dots , d$,
there is at least probability $c_6^d$ that $X^*$ is a
time-shifted path of $ X^j_t$ on an appropriate interval, for
all $ j = 1 , \dots , d$.

\medskip
\noindent{\it Step 4}. In this step, we will show that with a
positive probability, the process $K$ can have ``almost''
independent and ``almost'' perpendicular increments over
disjoint time intervals. Moreover, the distributions of the
increments have densities in an appropriate sense.

We find $d$ points $y_1, \dots, y_d \in \prt D$ such that
the $\n(y_j)$'s point in $d$ orthogonal directions. Let $C_j = \{\z
\in \RR^d: \angle(\n(y_j), \z) \leq \delta_0\}$, for some
$\delta_0>0$ so small that for any $\z_j \in C_j$, $j=1, \dots,
d$, the vectors $\{\z_j\}$ are linearly independent. Let
$\delta_1>0$ be so small that for every $j=1, \dots, d$, and
any $y \in \prt D \cap \BB(y_j, \delta_1)$, we have $\n(y) \in
C_j$.

Let $L^j$ be the local time of $X^j$ on $\prt D$ and $\sigma^j_t =
\inf\{s\geq 0: L^j_s \geq t\}$. It is easy to see that there
exists $p_2>0$ such that with probability greater than $p_2$, for
every $j = 1, \dots, d$, we have $X^j_t \notin \prt D \setminus
\BB(y_j, \delta_1)$, for $t\in [0, \sigma^j_1]$. Let
$$ R_j = \sup\{t< T_j: Y^j_t \in \prt D\} \qquad \hbox{and} \qquad
S_j = L^j_{R_j}.
$$
Let $F_*$ be the event that for every $j = 1, \dots, d$, $X^j_t
\notin \prt D \setminus \BB(y_j, \delta_1)$ for $t\in [0,
\sigma^j_1]$ and $S_j < \sigma^j_1$. Then (\ref{es5}) shows that
$\P_{x_1}(F_*) \geq p_2 (1- e^{-c_3})^d$.

Let $K^j_t = \Gamma \, \int_0^t \n(X^j_{\sigma^j_s}) \,ds$ and
note that if $F_*$ holds, then $K^j_t \in \Gamma \, C_j $ for all
$j=1, \dots, d$ and $t\in [0,1]$. Define for any $0\leq
a_k<b_k\leq k$ for $k=1, \cdots, d$,
 $$\Lambda([a_1,b_1], [a_2, b_2], \dots, [a_d,b_d])
 = \{ K^1_{t_1} + K^2_{t_2} + \dots + K^d_{t_d} :
 t_k \in [a_k, b_k] \hbox{  for  } 1\leq k \leq d \}.
 $$
It is easy to show using the definition of $C_j$'s that the
$d$-dimensional volume of $\Lambda([a_1,b_1], \dots,
[a_d,b_d])$ is bounded below by $c_7 \prod_{1\leq k \leq d}
(b_k - a_k)$, and bounded above by $c_8 \prod_{1\leq k \leq d}
(b_k - a_k)$.

Let us consider the processes defined above, conditioned on the
$\sigma$-field
 $${\cal G} =\sigma\left(\{K^j_t, t\in [0,1]\}_{1
 \leq j \leq d}\right).
 $$
By \eqref{es5} and \eqref{e:mu}, conditional on $\cal G$, the
random variables $S_j=L^j_{R_j}$, $j=1, \cdots, d$, have
distributions whose densities on $[0,1]$ are bounded below. In
view of our remarks on the volume of $\Lambda$, it follows that
conditional on $\cal G$, the vector $K^1_{S_1} + \dots +
K^d_{S_d}$ has a density with respect to $d$-dimensional Lebesgue
measure that is  bounded below by $c_9>0$ on a ball $U_*$ with
positive radius. We now remove the conditioning to conclude that
$K^1_{S_1} + \dots + K^d_{S_d}$ has a component with a density
with respect to $d$-dimensional Lebesgue measure that is  bounded
below on $U_*$.

Define $K^*_t = \Gamma \, \int_0^t \n(X^*s ) d\wt L_s$ and $T_* =
\sum_{j=1}^d T_j$, where $\wt L$ is the boundary local time for
reflecting Brownian motion $X^*$. Using conditioning on $F_*$, we
see that the distribution of $K^*_{T_*}$ has a component with
density greater than $c_9$ on $U_*$. Since $K^*$ does not change
when $X^*$ is inside the domain and $X^*$ is a reflecting Brownian
motion, we conclude that $(X^*_{T_* +1}, K^*_{T_* +1})$ has a
component with a density with respect to $2d$-dimensional Lebesgue
measure on a non-empty open set. It follows that for some fixed
$t_*>0$, $(X^*_{t_* }, K^*_{t_* })$ has a component with a density
with respect to $2d$-dimensional Lebesgue measure on a non-empty
open set.

We leave it to the reader to check that the argument can be easily
modified so that we can show that for any fixed $t_0>0$, $(X^*_{t_0/2
}, K^*_{t_0/2 })$ has a component with a strictly positive density
with respect to $2d$-dimensional Lebesgue measure on a non-empty open
set. We can now combine this with the result of Step 1 using the
Markov property to see that for some non-empty open set $\widetilde
U$ and any starting point $(X_0, K_0) = (x_0,y_0)$, the process
$(X_{t_0 }, K_{t_0 })$ has a positive density with respect to
$2d$-dimensional Lebesgue measure on $\widetilde U$ under
$\P_{x_0,y_0}$.

By the Girsanov theorem the same conclusion holds for $(X,K)$
under the measure $\Q_{x_0,y_0}$, which, by the proof of Theorem
\ref{T:2.1}, is the reflecting diffusion with inert drift. This
proves the theorem in the case $\rho \equiv  1$.

Now let us consider the general case when $\rho \in C^2(\overline D)$
is bounded between two positive constants. In this case, $(X, K)$ is
a solution to the following special case of \eqref{e:diffV1R}:
\begin{equ}[eq:137]
 \begin{cases}
dX_t=dB_t + \frac12 \nabla \log \rho (X_t) \,dt + \n (X_t) \,dt +K_t \,dt, \\
dK_t =\Gamma \,  \n (X_t) \,dL_t.
\end{cases}
\end{equ}
Let $\{(X, K), \, \P\}$ denote the solution to \eqref{eq:137},
and let $\{(X^0, K^0), \, \P^0\}$ correspond to the special
case $\rho \equiv 1$, with the same initial distribution. Note
that our assumptions on $\rho$ imply that $\nabla \log \rho \in
C^1(\overline D)$. Hence $\P$ and $\P^0$ are related by
Girsanov transform
$$ \frac{ d \P}{d\P^0} = \exp \left(\frac12  \int_0^t \nabla \log \rho (X^0_s)\,dB_s
 -\frac18 \int_0^2 | \nabla \log \rho (X_s^0)  |^2 \,ds \right) \qquad
 \hbox{on } \FF_t.
 $$
The Harris irreducibility of $(X, K)$ now follows from that for
$(X^0, K^0)$.
 \end{proof}

\begin{theorem}\label{T:irr.1}
Consider the SDE \eqref{e:diffV1R} with $\sigma$ being the identity
matrix, $\bfb = {1\over 2} \log \rho$, $\cn \equiv \n$ and $\rho \in
C^2(\overline D)$ being bounded between two positive constants. The
probability distribution $\pi(dx, dy)$ defined by
 $$\pi(A) =\int_A c_1
  \1_D (x) \rho (x) e^{-(\Gamma^{-1}y, y)} \, dx\, dy,
 \qquad A\subset \overline{D}\times \R^d,
 $$
is the only invariant measure for the solution $(X,K)$ to
\eqref{e:diffV1R}.
\end{theorem}

\begin{proof}
In view of Theorem \ref{T:wc.3}, all we have to show is
uniqueness. If there were more than one invariant measure, at
least two of them (say, $\mu$ and $\nu$) would be mutually
singular by Birkhoff's ergodic theorem \cite{Sinai}. However,
we have just shown that there exists a strictly positive
measure $\psi$ which is absolutely continuous with respect to
any transition probability, so that in particular, $\psi \ll
\mu$ and $\psi \ll \nu$. Since $\mu \perp \nu$ by assumption, there exists a
set $A$ such that $\mu(A) = 0$ and $\nu(A^c) = 0$. Therefore,
one must have $\psi(A) = \psi(A^c) = 0$ which contradicts the
fact that the measure $\psi$ is non-zero.
\end{proof}

\end{doublespace}

\vskip 0.6truein

\noindent {\bf Richard F. Bass:}

Department of Mathematics, University of Connecticut,  Storrs, CT
06269-3009, USA.

 Email: \texttt{bass@math.uconn.edu}

\bigskip

\noindent{\bf Krzysztof Burdzy:}

Department of Mathematics, University of Washington,  Seattle, WA
98195, USA.

 Email: \texttt{burdzy@math.washington.edu}

\bigskip

\noindent{\bf Zhen-Qing Chen:}

 Department of Mathematics, University of Washington, Seattle, WA 98195, USA.

  Email: \texttt{zchen@math.washington.edu}

\bigskip

\noindent{\bf Martin Hairer:}

 Mathematics Institute, The University of Warwick, Coventry CV4 7AL, UK.

  Email: \texttt{M.Hairer@Warwick.ac.uk}

\end{document}